\documentclass[10pt,a4paper]{article}

\usepackage{tikz}
\usepackage{amssymb,amsfonts,amsmath}
\usepackage[english]{babel}
\usepackage{amsthm}
\usepackage{graphicx}
\usepackage{caption}
\usepackage{subcaption}
\usepackage[]{algorithm2e}
\usepackage{authblk}
\usepackage{cite}
\usepackage{float}
\usepackage{pgfplots}
\pgfplotsset{ignore legend/.style={every axis legend/.code={}}}
\pgfplotsset{compat=newest}
\usetikzlibrary{plotmarks}
\usepackage{grffile}

\usepgfplotslibrary{external}
\tikzsetfigurename{tikzpdf_}
\tikzset{external/optimize command away=\AddToShipoutPicture}
\tikzset{external/system call={lualatex \tikzexternalcheckshellescape -halt-on-error
        -interaction=batchmode -jobname "\image" "\texsource"}}

\newcommand{\mb}[1]{\mathbf{#1}}
\newtheorem{theorem}{Theorem}[section]
\newtheorem{definition}[theorem]{Definition}
\newtheorem{example}[theorem]{Example}
\newtheorem{corollary}[theorem]{Corollary}
\newtheorem{remark}[theorem]{Remark}
\newtheorem{proposition}[theorem]{Proposition}

\numberwithin{equation}{section}

\usepackage{a4wide}

\begin{document}
\title{On a fast Arnoldi method for $BML$-matrices}
\author{Bernhard Beckermann, Clara Mertens, and Raf Vandebril}

\date{}

\maketitle

%\comment{Abstract added}
\abstract{Matrices whose adjoint is a low rank perturbation of a rational function of the matrix naturally arise when trying to extend the well known Faber-Manteuffel theorem \cite{FaMa84, FaLieTi}, which
provides necessary and sufficient conditions for the existence of a short Arnoldi recurrence.
We show that an orthonormal Krylov basis for this class of matrices can be generated by a short recurrence relation based on GMRES residual vectors. These residual vectors are computed
by means of an updating formula. Furthermore, the underlying Hessenberg matrix has an accompanying low rank structure, which we will investigate closely.}

\section{Introduction}
%\comment{Introduction has been partially rewritten, some paragraphs are reformulated in order to improve the flow of the introduction.}

In this article we will discuss a new variant of the Arnoldi method applied to a class of sparse matrices $A\in \mathbb C^{n\times n}$
which allows to compute the first $k$ Arnoldi
vectors in complexity $\mathcal O(kn)$. %\comment{Is this correct? I guess the complexity is $\mathcal O(kn)$ plus the complexity needed to do a matrix vector product?}
We will refer to this class of matrices as \textit{BML-matrices}, following the fundamental work of
Barth \& Manteuffel \cite{barthmanteuffel}
and Liesen \cite{Liesen} on matrices $A$ whose adjoint is a low rank perturbation of a rational function of $A$. %\comment{(Before submitting, need to make sure that we don't forget someone!)}
More specifically, we assume that
\begin{equation} A^\ast = p(A)q(A)^{-1} + FG^\ast, \label{eq:introeq} \end{equation}
with $p, q$ polynomials of degree $m_1$ and $m_2$, respectively, and
\[F = [\mb f_1, \mb f_2, \ldots, \mb f_{m_3}], \quad G=[\mb g_1, \mb g_2, \ldots, \mb g_{m_3}] \in \mathbb C^{n \times m_3} \]
matrices of full column rank. Moreover, it is assumed that the roots of $q$ are simple. By taking $p(z)/q(z)\in \{z,1/z\}$, we see that Hermitian matrices and unitary matrices
are $BML$-matrices, and the same is true for low rank perturbations of such matrices. Furthermore, if $FG^\ast = 0$ in \eqref{eq:introeq}, the matrix $A$ is normal \cite{Liesen}.
In what follows we suppose that
$m_j \ll n$, since these quantities (as well as the sparsity pattern of $A$) are hidden in the constant of the above-claimed complexity result.

After $k$ steps of the Arnoldi process with initial vector $\mb{b}$ one obtains the expression
 \begin{equation} AV_k = V_k H_k + h_{k+1,k} \mb{v}_{k+1} \mb{e}_k^\ast, \label{eq:arnoldistep} \end{equation}
 with $V_k = [\mb{v}_1, \mb{v}_2, \ldots, \mb{v}_k]$, $\mb{v}_1=\mb{b}/\|\mb{b}\|$, $\mb{v}_{k+1}\in \mathbb{C}^n$ and $h_{k+1,k} \geq 0$ satisfying $V_k^\ast V_k = I_k$
 and $V_k^\ast \mb{v}_{k+1} =0$. The matrix $H_k$ is upper Hessenberg.

A fast variant of the Arnoldi process will exploit additional structure of the upper Hessenberg matrix $H_k$. For example, to compute the successive vectors $\mb v_k$ for Hermitian $A$ we get a tridiagonal $H_k$
and the Arnoldi process reduces to the Lanczos method. For matrices $A$ satisfying \eqref{eq:introeq}, $H_k$ turns out to be a rank structured
matrix.
In order to specify this statement, the following definition is introduced.

%\comment{Definition 1.1 has been adapted so that it also includes the case where $r<0$, i.e., all submatrices have rank at most $s$ above and including the $-r$th subdiagonal.}
\begin{definition}\label{def:upper_separable}
   We say that a matrix $B\in \mathbb C^{n\times n}$ is $(r,s)-$upper-separable with $ s \geq 0$,
   if for all $j=1, 2, \ldots, n-|r|$ it holds that (in Matlab notation)
   \[
       \mbox{rank~} B(1:j-\min (r,0),j+\max (r,0):n)\leq s.
   \] In other words, any submatrix of $B$ including elements on and above the $r$th diagonal of $B$ is of rank at most $s$. The $0$th diagonal corresponds to the main
   diagonal, while the $r$th diagonal refers to the $r$th superdiagonal if $r >0$ and to the $-r$th subdiagonal if $r<0$.
\end{definition}

Before proceeding, we give some comments on Definition~\ref{def:upper_separable} whose formulation is inspired by some related well-established definitions. For example, matrices with both $B$ and $B^*$
being $(0,1)$--upper-separable (and $(1,1)$--upper-separable, respectively) are referred to as semiseparable matrices (and quasi separable, respectively) \cite[\S1.1, \S 9.3.1]{b163}.
As a simple example, a tridiagonal matrix is quasi separable, and its inverse is known to be semiseparable. Matrices being $(1-s,s)$--upper-separable with their adjoint being $(1-r,r)$--upper-separable are usually
called $(r,s)-$semiseparable, while matrices being $(1,s)$--upper-separable with their adjoint being $(1,r)$--upper-separable are
called $(r,s)-$quasi separable \cite[\S8.2.2 and \S8.2.3]{b163}.

% \comment{I added $r, s \geq 0$ to the above definition. We could also extend the definition to $r < 0$, but this would make the notation more complicated and we only encounter upper-separable
% matrices with $r \geq 0$ in the rest of the article. Raf also proposes to just call it separable instead of upper-separable.}
The above statement on the rank structure of $H_k$ can now be made exact. Assume the Arnoldi process breaks down after $N \leq n$ iterations, i.e., $\mb{v}_{N+1} = 0$ in \eqref{eq:arnoldistep}. We will refer
to $N$ as the `Arnoldi termination index' in the rest of the article. In
Corollary \ref{cor:semisep3} it is shown that for a $BML$-matrix $A$, the underlying Hessenberg matrix $H_N$ is
$(r,s)-$upper-separable, where $r$ and $s$ are functions of $m_j$. However, for a general matrix $A$, there is no reason for the underlying Hessenberg matrix to be upper-separable.
The upper-separable structure of the underlying Hessenberg matrix gives rise to the design of
a \textit{multiple recurrence relation} \cite{barthmanteuffel,mertensvandebril}, signifying that each new Arnoldi vector can be written as
\begin{equation}
 \mb{v}_{k+1} = \sum_{j=k-m_3-m_2}^k \alpha_{k,j}A \mb{v}_j + \sum_{j=k-m_3-m_1}^k \beta_{k,j} \mb{v}_j. \label{eq:multi}
\end{equation}
The smaller the quantities $m_j$ in \eqref{eq:introeq}, the shorter the recurrence relation becomes. In \cite{bella} the same recurrence relations are derived for the class of so called
$(H,m)$-well-free matrices. We refer to Remark \ref{remark:wellfree} for some more details.

In this article we investigate a different version of the recurrence relation \eqref{eq:multi} by rewriting it in terms of GMRES residual vectors, aiming to overcome some
of the numerical problems which relation \eqref{eq:multi} entails, such as the possibility of a breakdown \cite{mertensvandebril}. It will be shown
how these GMRES residual vectors can be computed progressively by means of an updating formula.
This partially extends the discussion on a progressive GMRES method for nearly Hermitian matrices as presented by Beckermann \& Reichel \cite{BeckReich}.

The article is organized as follows. Section \ref{sec:proggmres} describes the structure of the unitary factor $Q$ in the $QR$-decomposition of a Hessenberg matrix in terms of orthogonal polynomials
(the results in this section are valid for a general matrix $A \in \mathbb{C}^{n \times n}$).
It is well known that this unitary factor can be represented as a product of Givens rotations, see e.g.,
the isometric Arnoldi process introduced by Gragg \cite{m281} and its extension to the class of shifted unitary matrices by Jagels \& Reichel \cite{JaRe94}. We will
describe these Givens rotations by means of orthogonal polynomials.
Furthermore, links between orthogonal polynomials and the GMRES algorithm are discussed, leading to an updating
formula to compute the GMRES residuals progressively.
In section \ref{sec:rec} the upper-separable structure of the Hessenberg matrix related to the Arnoldi process applied to a $BML$-matrix is investigated. It
is shown how this upper-separable structure can be generated by the GMRES residual vectors, allowing to construct a short recurrence relation and an
accompanying algorithm. In section \ref{sec:barthmanteuffel} we compare our findings with those presented in \cite{barthmanteuffel}. Section \ref{sec:unitary} discusses
some computational reductions that can be made in case the matrix $A$ is nearly unitary or nearly shifted unitary.
Finally, section \ref{sec:num} discusses the numerical performance and stability of the algorithm.

 Throughout this article we will make use of the following notation.
Vectors are written in
bold face lower case letters, e.g., $\mb{x},\mb{y}$, and $\mb{z}$. The vector $\mb{e}_k$ denotes the $k$th column of an identity matrix of
applicable order. The standard inner product is denoted as $\langle \mb{x},\mb{y} \rangle = \mb{y}^\ast \mb{x}$, with $\cdot^\ast$ the Hermitian conjugate.
We write $\| \cdot \|$ for the induced Euclidean norm as well as the subordinate spectral matrix norm. Matrices are denoted by upper case
letters $A=(a_{ij})$, and $I_k$ denotes the identity matrix of order $k$.
We will frequently express formula \eqref{eq:arnoldistep} in the form
\begin{equation*}
   AV_k = V_{k+1}\underline{H}_k, \quad \text{with}\quad V_{k+1} := [V_k,\mb{v}_{k+1}],
\end{equation*}
where
\[
   \underline{I}_k
   : =  \left[\begin{array}{c}
    I_k \\
                                                            \mb{0}
                                                           \end{array}\right] \in \mathbb{C}^{(k+1) \times k}
    , \quad
    \underline{H}_k := \left[\begin{array}{c}
                                                         H_k \\
                                                         h_{k+1,k}\mb{e}_k^\ast
                                                        \end{array}\right] = H_{k+1} \underline I_k \in \mathbb{C}^{(k+1) \times k}
    ,
\]
and $V_k = V_{k+1} \underline I_k$,
revealing the nested structure of the Arnoldi matrices $V_k$ and $H_k$.

\section{Towards a progressive GMRES method} \label{sec:proggmres}
%\comment{Next two sections: only small changes made.}\\
Many Krylov space methods and in particular the Arnoldi method can be described in polynomial language which reveals some particular properties, and makes
the link to the rich theory of orthogonal polynomials. For example, from \eqref{eq:arnoldistep} one deduces by recurrence on the degree that, for any polynomial $p$
 \begin{equation} \label{eq:exactness}
     p(A) \mb v_1 = V_k p(H_k) \mb e_1, \quad \mbox{provided that $\deg p < k$},
\end{equation}
illustrating that Krylov spaces are intimately related to polynomials. See e.g., \cite[\S6.6.2]{saad} for the case of Hermitian $A$. However, polynomial language can
be used also for general matrices \cite[\S1.3]{Beck}.
In \S \ref{subsec:OP} we will see that the Arnoldi vectors correspond to a (finite) family of polynomials which are orthogonal with respect to
\begin{equation} \label{eq:scalar_product}
\langle p, q \rangle_{A,\mb v_1} = \langle p(A)\mb{v}_1, q(A)\mb{v}_1 \rangle = \left( q(A)\mb{v}_1\right)^\ast p(A)\mb{v}_1,
\end{equation}
a scalar product on the set of polynomials of degree $<N$, with $N$ the termination index of the Arnoldi method. The scalar product \eqref{eq:scalar_product} induces a norm $\| \cdot \|_{A,\mb v_1}$.
This implies in particular the well known fact \cite[Proposition~6.7]{saad} that Arnoldi vectors are normalized FOM residuals.
In \S \ref{subsec:QR} we will use polynomial language to give an explicit expression for the $Q$-factor in a $QR$-decomposition of an upper Hessenberg matrix. Such a formula for a unitary upper
Hessenberg matrix is not known in literature, though of course there is a close link with Gragg's explicit formula in terms of Givens rotations \cite[\S6.5.3]{saad,m281}.
This formula will enable us to deduce in Corollary \ref{cor:semisep} a decay property of the entries of $Q$ far
from the main diagonal, and reveals immediately a rank structure for $Q$. In \S\ref{subsec:GMRES} we recall that GMRES residuals can be expressed in terms of orthogonal polynomials: we are faced
with a well-studied extremal problem for general orthogonal polynomials. This allows us in \S \ref{sec:prog} to establish a well known link between (normalized) FOM and GMRES
residuals \cite[\S6.5.5]{saad}, allowing for a recursive computation of (normalized) GMRES residuals. Such a progressive GMRES implementation has been discussed before
\cite[\S6.5.3]{saad} and \cite{BeckReich}. The implementation presented in this article is inspired by the work of Beckermann \& Reichel \cite{BeckReich}. Both implementations make use
of the decomposition of $Q$ into a product of Givens rotations, but differ in how to find the angles of these rotations. We will consider in \S \ref{sec:prog} not only FOM and GMRES for systems of
linear equations $A\mb x=\mb b$  but more generally for shifted systems $(A-\delta I)\mb x=\mb b$ for some parameters $\delta\in \mathbb C$. All findings of this section hold for general matrices $A$.

 \subsection{Orthogonal polynomials linked to the Arnoldi process}\label{subsec:OP}

Given an $N\times N$ upper Hessenberg matrix $H_N$ with positive real entries $h_{k+1,k}$ on the subdiagonal.
We define polynomials $q_0, q_1,...,q_{N-1}$ recursively through the formula
\begin{eqnarray}
       q_0(z) &=& 1, \nonumber \\
       q_k(z) h_{k+1,k} &=& z q_{k-1}(z) - \sum_{j=1}^k q_{j-1}(z) h_{j,k}, \,\, \text{if}\,\, N-1 \geq k \geq 1. \label{eq:qrec}
\end{eqnarray}
 Direct computation yields the following well known link with the characteristic polynomials of the principal submatrices:
 \begin{equation} q_k(z) = \frac{1}{\prod_{i=1}^k h_{i+1,i}}\text{det}(zI_k-H_k), \label{eq:qdet} \end{equation}
showing that $q_k$ is of degree $k$, with positive leading coefficient. In what follows we will write \eqref{eq:qrec} in the form
\begin{equation}
   (q_0(z), \ldots, q_{k}(z)) \underline H_k = z(q_0(z), \ldots, q_{k-1}(z)),
 \label{eq:1}
\end{equation}
where $H_k$ is the $k \times k$ principal minor of $H_N$ and $N-1 \geq k \geq 1$.
One deduces from \eqref{eq:qrec} by recurrence on $k$ that $q_k(H_N)\mb{e}_1=\mb{e}_{k+1}$. Assume we apply the Arnoldi process to a matrix $A \in \mathbb{C}^{n \times n}$, that $N \leq n$ is
the Arnoldi termination index and $H_N$ the underlying Hessenberg matrix. Using \eqref{eq:exactness}, this implies
\begin{equation} \label{eq:arnoldi}
       \mb v_k = q_{k-1}(A)\mb v_1,
\end{equation}
and thus indeed the $q_j$ is the $j$th orthonormal polynomial with respect to the scalar product \eqref{eq:scalar_product}. Also, from \eqref{eq:qdet} we see that the $k$th
FOM iterate for the shifted system $(A-\delta I)\mb x=\mb b$ exists if and only if $q_k(\delta)\neq 0$, and that in this case the $k$th FOM residual is given by
$q_k(A)\mb b/q_k(\delta)$, compare with \cite[Proposition~6.7]{saad}.
% We recursively define the sequence of polynomials $\{q_k\}$ as
%for $k\geq 1$.
%The orthonormal vectors $\mb{v}_i$ can be expressed as $\mb{v}_{i+1} = q_i(A)\mb{v}_1$. Moreover,
% the polynomials $q_k$ are orthonormal with respect to the inner product relation \cite[Section 6.6.2]{saad}
% \[\langle q_i, q_j \rangle_{L^2} = \langle q_i(A)\mb{v}_1, q_j(A)\mb{v}_1 \rangle = \left(q_j(A)\mb{v}_1\right)^\ast q_i(A)\mb{v}_1. \]

 \subsection{The $QR$-factorization of a Hessenberg matrix}\label{subsec:QR}
We will derive an explicit formula for the unitary factor in the $QR$-decomposition of the upper Hessenberg matrix $\underline H_k$ in terms of the orthonormal polynomials $q_0,...,q_k$.
To our knowledge, such a result is new. It could also be potentially useful for studying the convergence of the $QR$-method with shifts.

%\comment{I made some changes to Proposition 2.2 below. In particular, I made Remark 2.3 of the previous version part of Proposition 2.2 and some small additions are made to the proof.
%Also, $Q_{k+1}(\delta)$ is not really the ``unique'' unitary factor in the $QR$-decomposition of $(\underline H_k - \delta \underline I_k)$ as the latter matrix is not square. Only the first $k$ columns of $Q_{k+1}(\delta)$
%are unique, while the last column is only determined up to a unimodular constant. Therefore I introduced Definition 2.1, introducing the nested substructure which we use, and imposing
%the unitary factor to have determinant 1, which fixes the last column.}

Let $Q_{k+1}(\delta)$ be the unitary factor in the $QR$-decomposition of $\underline H_k - \delta \underline I_k$, i.e.,
   \begin{equation}
         Q_{k+1}(\delta)^\ast (\underline H_k - \delta \underline I_k)
         =   \underline{R}_k(\delta) := \left[\begin{array}{c}
                    R_k(\delta) \\
                    0
                   \end{array}\right] \in \mathbb{C}^{(k+1)\times k}, \label{eq:qrdef}
   \end{equation}
   $R_k(\delta)$ an upper triangular matrix with positive real entries on its main diagonal.
It is well known (see, e.g., \cite[Subsection~6.5.3]{saad}) that $Q_{k+1}(\delta)$ can be obtained
as a product of Givens rotations, which are applied to $\underline{H}_k-\delta \underline I_k$ to annihilate the first subdiagonal.
We follow \cite{BeckReich}, imposing the matrices $Q_{k+1}(\delta)$ to have determinant 1.
\begin{definition} \label{def:givens}
 Let $Q_1 = [1]$, and define for $k \geq 1$,
\[
  Q_{k+1}(\delta)^\ast = \Omega_{k+1}(\delta) \left[\begin{array}{cc}
                                    Q_k(\delta)^\ast & 0\\
                                    0 & 1
                                   \end{array}\right],
\quad \Omega_{k+1}(\delta) = \left[\begin{array}{crr}
                       I_{k-1} & 0 &0 \\
                       0 & \overline{c_k(\delta)} & s_k(\delta) \\
                       0 & -s_k(\delta) & c_k(\delta)
                      \end{array}\right], %\label{eq:givens}
\]
with $s_k(\delta) \geq 0$ and $s_k(\delta)^2 + |c_k(\delta)|^2 = 1$, such that \eqref{eq:qrdef} holds.
\end{definition}

\begin{proposition}\label{prop:QR}
   Let $\delta\in \mathbb C$, and
   \[
        \sigma_k(z) := \sqrt{\sum_{j=0}^k |q_j(z)|^2}.
   \]
   Then the unitary factor $Q_{k+1}(\delta)^\ast$
   is given by
   \[
      %{\small
      \left[\begin{array}{cccccc}
               -\frac{q_0(\delta)\overline{q_{1}(\delta)}}{\sigma_{1}(\delta) \sigma_{0}(\delta)} & \frac{\sigma_{0}(\delta)}{\sigma_1(\delta)}
               & 0  & \cdots & \cdots & 0 \\ [0.2cm]
             -\frac{q_0(\delta)\overline{q_{2}(\delta)}}{\sigma_{2}(\delta) \sigma_{1}(\delta)} &
             -\frac{q_1(\delta)\overline{q_{2}(\delta)}}{\sigma_{2}(\delta) \sigma_{1}(\delta)} &
             \frac{\sigma_{1}(\delta)}{\sigma_{2}(\delta)} & 0 & \\ [0.2cm]
              \vdots & \vdots & \ddots   & \ddots & \ddots & \vdots \\ [0.2cm]
              \vdots & \vdots & & \ddots   & \ddots & 0 \\ [0.2cm]
             -\frac{q_0(\delta)\overline{q_{k}(\delta)}}{\sigma_{k}(\delta) \sigma_{k-1}(\delta)} &
             -\frac{q_1(\delta)\overline{q_{k}(\delta)}}{\sigma_{k}(\delta) \sigma_{k-1}(\delta)} & \cdots& \cdots
             & -\frac{q_{k-1}(\delta)\overline{q_{k}(\delta)}}{\sigma_{k}(\delta) \sigma_{k-1}(\delta)} & \frac{\sigma_{k-1}(\delta)}{\sigma_k(\delta)} \\ [0.2cm]
            (-1)^k \frac{q_0(\delta)}{\sigma_k(\delta)} &  (-1)^k\frac{q_1(\delta)}{\sigma_k(\delta)} &  \cdots &  \cdots & (-1)^k\frac{q_{k-1}(\delta)}{\sigma_k(\delta)} &  (-1)^k\frac{q_k(\delta)}{\sigma_k(\delta)}
            \end{array}\right].
            %} %\label{eq:unistruct}
\] Moreover,
\begin{equation}
   s_k(\delta) = \frac{\sigma_{k-1}(\delta)}{\sigma_k(\delta)}\quad \text{and} \quad c_k(\delta) = (-1)^k \frac{q_k(\delta)}{\sigma_k(\delta)}, \quad \text{for} \,\, k \geq 1. \label{eq:cs}
 \end{equation}
\end{proposition}
\begin{proof}
   We leave it to the reader to check that the candidate for $Q_{k+1}(\delta)^\ast$
   indeed has orthonormal rows. It remains to check the subdiagonal and diagonal entries of $Q_{k+1}(\delta)^*(\underline H_k - \delta \underline I_k)$.
   For $k \geq j > \ell$ we find that
   \begin{eqnarray*} &&
      \mb e_j^\ast Q_{k+1}(\delta)^\ast(\underline H_k - \delta \underline I_k) \mb e_\ell
      \\&&= - \frac{\overline{q_j(\delta)}}{\sigma_j(\delta) \sigma_{j-1}(\delta)}\bigl(q_0(\delta), \ldots, q_{j-1}(\delta), \underbrace{0, \ldots, 0}_{k+1-j})
      (\underline H_k - \delta \underline I_k) \mb e_\ell
      \\&& + \left(\underbrace{0, \ldots, 0}_{j}, \frac{\sigma_{j-1}(\delta)}{\sigma_j(\delta)}, \underbrace{0, \ldots, 0}_{k-j} \right)(\underline H_k - \delta \underline I_k) \mb e_\ell
      \\&&= - \frac{\overline{q_j(\delta)}}{\sigma_j(\delta) \sigma_{j-1}(\delta)}
       \, \left( q_0(\delta),\ldots,q_{\ell}(\delta)\right)
      (\underline H_\ell-\delta \underline I_\ell)\mb{e}_\ell  = 0,
   \end{eqnarray*}
   where the second equality is because of the fact that only the first $\ell+1\leq j$ entries of $ (\underline H_k - \delta \underline I_k) \mb e_\ell$ are nonzero and
   the third equality because of \eqref{eq:1}. Similarly, for $k \geq j=\ell $,
   \begin{eqnarray*} &&
   \mb e_\ell^\ast Q_{k+1}(\delta)^*(\underline H_k - \delta \underline I_k) \mb e_\ell \\&&=
      - \frac{\overline{q_\ell(\delta)}}{\sigma_\ell(\delta) \sigma_{\ell-1}(\delta)}
      (q_0(\delta),\ldots,q_{\ell-1}(\delta), \underbrace{0, \ldots, 0}_{k+1-\ell})(\underline H_k - \delta \underline I_k) \mb e_\ell \\&&+
      \left(\underbrace{0, \ldots, 0}_{\ell}, \frac{\sigma_{\ell-1}(\delta)}{\sigma_\ell(\delta)}, \underbrace{0, \ldots, 0}_{k-\ell} \right)(\underline H_k - \delta \underline I_k) \mb e_\ell
      \\&&=
      - \frac{\overline{q_\ell(\delta)}}{\sigma_\ell(\delta) \sigma_{\ell-1}(\delta)}
      (q_0(\delta),\ldots,q_{\ell-1}(\delta))
      (H_\ell - \delta I_\ell) \mb e_\ell
      + \frac{\sigma_{\ell-1}(\delta)}{ \sigma_{\ell}(\delta)} h_{\ell+1,\ell} \\&&=
      - \frac{\overline{q_\ell(\delta)}}{\sigma_\ell(\delta) \sigma_{\ell-1}(\delta)} \left(-q_{\ell}(\delta) h_{\ell+1,\ell} \right) +
      \frac{\sigma_{\ell-1}(\delta)}{\sigma_{\ell}(\delta)} h_{\ell+1,\ell}  \\&&=
      \frac{h_{\ell+1,\ell}}{\sigma_\ell(\delta) \sigma_{\ell-1}(\delta)} \left(|q_\ell(\delta)|^2 + \sigma_{\ell-1}(\delta)^2 \right)
      = h_{\ell+1,\ell} \frac{\sigma_{\ell}(\delta)}{ \sigma_{\ell-1}(\delta)}>0.
   \end{eqnarray*}
   Finally, for $k+1 = j \geq \ell$,
   \begin{eqnarray*} &&
    \mb{e}_{k+1}^\ast Q_{k+1}(\delta)^\ast (\underline H_k - \delta \underline I_k) \mb e_\ell \\ &&=
    \frac{(-1)^k}{\sigma_k(\delta)} \left(q_0(\delta), q_1(\delta), \ldots, q_k(\delta) \right) \left(\underline H_k - \delta I_k \right) \mb e_\ell = 0,
   \end{eqnarray*}
   according to \eqref{eq:1}.
   To prove \eqref{eq:cs}, observe that
   \[
   Q_{k+1}(\delta)^\ast = \Omega_{k+1}(\delta) \left[\begin{array}{cc}
                                    Q_k(\delta)^\ast & 0\\
                                    0 & 1
                                   \end{array}\right],
    \]
if and only if by multiplying on the left with $\left[\begin{array}{rr}
                                    \overline{c_k(\delta)} & s_k(\delta)\\
                                    -s_k(\delta) & c_k(\delta)
                                   \end{array}\right]$ we transform
\[
\left[\begin{array}{cccccc}
            (-1)^{k-1} \frac{q_0(\delta)}{\sigma_{k-1}(\delta)} &  (-1)^{k-1}\frac{q_1(\delta)}{\sigma_{k-1}(\delta)} &  \cdots &  (-1)^{k-1}\frac{q_{k-1}(\delta)}{\sigma_{k-1}(\delta)} &  0
            \\ [0.2cm] 0 & 0 & \cdots & 0 & 1
            \end{array}\right]
\]
into
\[
      \left[\begin{array}{cccccc}
             -\frac{q_0(\delta)\overline{q_{k}(\delta)}}{\sigma_{k}(\delta) \sigma_{k-1}(\delta)} &
             -\frac{q_1(\delta)\overline{q_{k}(\delta)}}{\sigma_{k}(\delta) \sigma_{k-1}(\delta)} & \cdots
             & -\frac{q_{k-1}(\delta)\overline{q_{k}(\delta)}}{\sigma_{k}(\delta) \sigma_{k-1}(\delta)} & \frac{\sigma_{k-1}(\delta)}{\sigma_k(\delta)} \\ [0.2cm]
            (-1)^k \frac{q_0(\delta)}{\sigma_k(\delta)} &  (-1)^k\frac{q_1(\delta)}{\sigma_k(\delta)} &  \cdots &  (-1)^k\frac{q_{k-1}(\delta)}{\sigma_k(\delta)} &  (-1)^k\frac{q_k(\delta)}{\sigma_k(\delta)}
            \end{array}\right],
\]
the latter being true for $c_k(\delta)$ and $s_k(\delta)$ as in \eqref{eq:cs}.
\end{proof}

Notice that, according to the nested structure of the Hessenberg matrices,
$Q_{k+1}(\delta)^{\ast}(H_{k+1}-\delta I_{k+1})$ is also upper triangular,
but its last diagonal entry given by
\[
           (-1)^{k+1} \, {h_{k+2,k+1} \, q_{k+1}(\delta)} / {\sigma_k(\delta)}
\]
is not necessarily a positive real number. Hence, for obtaining the unitary factor in the unique $QR$-decomposition of $H_{k+1}-\delta I_{k+1}$ we should rescale
the last row of $Q_{k+1}(\delta)^\ast$ as given in Proposition~\ref{prop:QR} by a phase of modulus $1$.

Let us consider the special case of $\delta=0$ and unitary $A$ as a running example.

\begin{example}\label{ex:unitary1}
    Suppose that $\delta=0$ and $A$ is unitary.
    Then $A V_k=V_{k+1} \underline H_k$ and thus $\underline H_k$ has orthonormal columns, showing that
    $\underline H_k =Q_{k+1}(0) \underline I_k$ and $R_k(0)=I_k$.
    From Proposition~\ref{prop:QR} we get explicit formulas for the entries of $\underline H_k$, in particular, for unitary $A$,
    \begin{equation} \label{eq:some_entries}
         h_{k+1,k} = \frac{\sigma_{k-1}(0)}{\sigma_k(0)} = s_k(0),
         \quad
         (-1)^{k-1} \sigma_{k-1}(0) h_{1,k}=
         (-1)^{k}  \frac{q_k(0)}{\sigma_k(0)} = c_k(0).
    \end{equation}
Furthermore, $A V_k = V_{k+1} Q_{k+1}(0) \underline I_k$.
Therefore,
\begin{eqnarray*}
 A \mb v_k &=& V_{k+1} \left[ \begin{array}{cc} Q_k(0) & 0\\ 0 & 1 \end{array} \right] \left[ \begin{array}{cc} I_{k-1} & 0 \\ 0 & c_k(0) \\ 0 & s_k(0) \end{array} \right] \mb e_k \\
           &=& V_{k+1} \left[ \begin{array}{c} c_k(0) Q_k(0)\mb e_k \\ s_k(0) \end{array} \right] \\
           &=& s_k(0) \mb v_{k+1} + (-1)^{k-1} c_k(0) \mb{\tilde v}_k,
\end{eqnarray*}
with $\mb{\tilde v}_k := \frac{1}{\sigma_{k-1}(0)}\sum_{j=0}^{k-1} \overline{q_j(0)} \mb v_{j+1}$. Also, $\mb {\tilde v}_{k+1} = s_k(0) \mb {\tilde v}_k + (-1)^k \overline{c_k(0)} \mb v_{k+1}$.
Hence, the orthonormal vectors $\mb v_k$ can be constructed using two short recurrence relations:

\begin{algorithm*}[H]
 $\mb{\tilde v}_1 := \mb v_1$\;
  \For{$k=1, \ldots, N-1$}{
  $\mb z:= A \mb v_k$\;
  $\gamma_k := -\tilde{\mb v}_k^\ast \mb z$\;
  $\sigma_k := (1-|\gamma_k|^2)^{1/2}$\;
  $\mb v_{k+1} := \sigma_k^{-1}\left(\mb z + \gamma_k \tilde{\mb v}_k \right)$\;
  $\mb{\tilde v}_{k+1} := \sigma_k \mb{\tilde v}_k + \overline{\gamma}_k \mb v_{k+1}$\;
  }
\end{algorithm*}
Note that $\gamma_k = (-1)^k c_k(0)$ and $\sigma_k = s_k(0)$.
The above double recurrence relation is known as the `Isometric Arnoldi algorithm' designed by Gragg \cite{m281}.

\end{example}

As a consequence of Proposition~\ref{prop:QR}, according to Definition~\ref{def:upper_separable}, we can derive some statements on the rank structure of the unitary
Hessenberg matrix $Q_{k+1}(\delta)$, and on a decay property of its entries.

\begin{corollary}\label{cor:semisep}
   $Q_{k+1}(\delta)$ is $(0,1)$-upper-separable\footnote{This implies that $Q_{k+1}(\delta)$ is quasi separable, but in general not semiseparable.}.
    Moreover, for the submatrix $\widetilde Q$ of $Q_{k+1}(\delta)$ formed with the first $m \leq k+1$ rows and the last $\ell \leq k-m+2$ columns we have that
   \[
       \| \widetilde Q \| = \sigma_{m-1}(\delta)/\sigma_{k-\ell+1}(\delta).
   \]
 %\comment{The above statement is valid for $\ell \leq k-m+2$? Or am I incorrect? In the previous version it was stated $\ell \leq k -m$.}
\end{corollary}
\begin{proof}
   The first statement follows by observing that
  \begin{equation} Q_{k+1}(\delta) =  \left[\begin{array}{c}
              \overline{q_0(\delta)} \\
              \vdots \\
              \overline{q_k(\delta)}
             \end{array}
 \right] \mb{e}_1^\ast Q_{k+1}(\delta) +  L_{k+1}(\delta),  \label{eq:structQ} \end{equation}
  with $L_{k+1}(\delta)$ strictly lower triangular, i.e., with zero entries on the main diagonal. This is a direct consequence of Proposition~\ref{prop:QR}.
  In particular, we deduce that $\widetilde Q$ is of rank $1$.
  %\comment{I added ``or by recurrence on the number of
  %Givens rotations'', because I think it is an interesting insight to know that the upper-separable structure is an immediate consequence of the fact that $Q_k(\delta)$ can be
  %factored as a product of Givens rotations. Otherwise it looks like one can only know that $Q_k(\delta)$ has an upper-separable structure by deriving an explicit formula for
  %its entries, while you actually know in advance that it should be $(0,1)$-upper-separable. Maybe this is not the best way to state it, but I think we should at least
  %mention it somewhere.}
 From Proposition~\ref{prop:QR} it follows that
  \begin{eqnarray*}
      \| \widetilde Q \|^2  &=&
      \frac{\sigma_{m-1}(\delta)^2}{\sigma_k(\delta)^2}
      + \sum_{j=k-\ell+2}^k \frac{|q_j(\delta)|^2}{\sigma_j(\delta)^2 \sigma_{j-1}(\delta)^2} \sigma_{m-1}(\delta)^2 \\
      &=& \frac{\sigma_{m-1}(\delta)^2}{\sigma_k(\delta)^2} + \sum_{j=k-\ell+2}^k \frac{\sigma_j(\delta)^2 - \sigma_{j-1}(\delta)^2}{\sigma_j(\delta)^2 \sigma_{j-1}(\delta)^2} \sigma_{m-1}(\delta)^2 \\
      &=& \sigma_{m-1}(\delta)^2 \left(\frac{1}{\sigma_k(\delta)^2} + \sum_{j=k-\ell+2}^k \left( \frac{1}{\sigma_{j-1}(\delta)^2} - \frac{1}{\sigma_j(\delta)^2} \right) \right),
  \end{eqnarray*}
  giving the claimed result.
\end{proof}

We end this subsection by observing that Corollary~\ref{cor:semisep} immediately  implies a rank property as well as a decay of entries for resolvents of $H_{k+1}$.

\begin{corollary}\label{cor:semisep2}
   Suppose that $H_{k+1}-\delta I_{k+1}$ is invertible. Then
   $(H_{k+1}-\delta I_{k+1})^{-*}$ is $(0,1)$-upper-separable.

   Moreover, for the submatrix $\widetilde H$ of $(H_{k+1}-\delta I_{k+1})^{-*}$ formed with the first $m \leq k+1$ rows and the last $\ell \leq k-m+2$ columns we have that
   \[
       \| \widetilde H \| \leq \| (H_{k+1}-\delta I_{k+1})^{-1} \| \,  \sigma_{m-1}(\delta)/\sigma_{k-\ell+1}(\delta).
   \]
   %\comment{Again I changed $\ell \leq k-m$ to $\ell \leq k-m+2$. Also, some changed are made to the proof below, hopefully to improve its readability.}
\end{corollary}
\begin{proof}
   Let us write $H_{k+1}-\delta I_{k+1}=Q_{k+1}(\delta) R$ with upper triangular and invertible $R$. Then
   \begin{equation*}
        (H_{k+1}-\delta I_{k+1})^{-*} = Q_{k+1}(\delta) R^{-*}
   \end{equation*}
   with $R^{-*}$ lower triangular of norm $\| (H_{k+1}-\delta I_{k+1})^{-1} \|$. Replacing $Q_{k+1}(\delta)$ by \eqref{eq:structQ} yields
   \begin{equation}  \label{eq:resolvent}
      (H_{k+1}-\delta I_{k+1})^{-*} = (q_0(\delta),\ldots,q_k(\delta))^\ast \, \mb e_1^\ast (H_{k+1}-\delta I_{k+1})^{-\ast} + \widetilde L_{k+1} (\delta),
   \end{equation}
   with $\widetilde L_{k+1}(\delta)$ strictly lower triangular; proving the first statement. This implies that
   \begin{equation*}
      \widetilde H = (q_0(\delta),...,q_{m-1}(\delta))^\ast \, \mb e_1^\ast
      \widetilde H,
   \end{equation*}
   and thus $\| \widetilde H \| = \sigma_{m-1}(\delta) \, \| \mb e_1^\ast
      \widetilde H \|$.
   Notice that $\mb e_1^\ast
      \widetilde H$ is obtained by multiplying the first row of $Q_{k+1}(\delta)$ with the last $\ell$ columns of $R^{-*}$. As $R^{-\ast}$ is lower triangular, $\mb{e}_1^\ast \widetilde H = \widetilde Q \widetilde R$,
      $\widetilde Q$ a row vector formed with the last $\ell$ entries of the first row of $Q_{k+1}(\delta)$ and $\widetilde R$ the lower-right $\ell \times \ell$ minor of $R^{-\ast}$.
      Therefore, applying Corollary~\ref{cor:semisep} to $\widetilde Q$ yields
      \[
         \| \mb e_1^\ast \widetilde H \|
         \leq
         \, ||\widetilde Q||\, ||\widetilde R|| \leq
         \frac{\sigma_0(\delta)}{\sigma_{k-\ell+1}(\delta)} ||R^{-\ast}||,
      \]
      which together with  $\| \widetilde H \| = \sigma_{m-1}(\delta) \, \| \mb e_1^\ast
      \widetilde H \|$ and $\|R^{-*}\| = \| (H_{k+1}-\delta I_{k+1})^{-1} \|$ proves the second statement.
\end{proof}

 \subsection{The GMRES residual and orthogonal polynomials}\label{subsec:GMRES}
We will give some more details on the link between the GMRES residual vectors and orthogonal polynomials. More specifically, we will write the GMRES residual
vector as a linear combination of Arnoldi vectors.
In the particular case of unitary $A$, an even nicer relation arises for the GMRES residual vectors.

In what follows we denote by $\mb r_k(\delta)$ the $k$th GMRES residual for the shifted system $(A-\delta I_n)\mb x=\mb b$, with starting vector $\mb{x}_0=0$,
and denote by $\mb w_k(\delta):=\mb r_k(\delta)/\| \mb r_k(\delta)\|$ its normalized version.

\begin{proposition}
   The $k$th GMRES residual for the shifted system $(A-\delta I_n)\mb x=\mb{b}$ with starting vector $\mb{x}_0=0$, can be written as
   \begin{equation}
       \mb{r}_k(\delta)= p_k(A)\mb{r}_0(\delta) , \quad \mbox{where}\quad p_k(z) = \frac{1}{\sigma_k^2(\delta)} \sum_{j=0}^k \overline{q_j(\delta)}q_j(z). \label{eq:respol}
   \end{equation}
   For its normalized version we have
   \begin{equation}
      \mb{w}_k(\delta) = \frac{\mb{r}_k(\delta)}{\|\mb{r}_k(\delta)\|} = \frac{1}{\sigma_k(\delta)}\sum_{j=0}^k \overline{q_j(\delta)}\mb{v}_{j+1}. \label{eq:respolnorm}
   \end{equation}
\end{proposition}
\begin{proof}
   Since we choose as starting vector $\mb{x}_0=0$, we find the initial GMRES residual $\mb r_0(\delta) = \mb b - (A-\delta I_n)0=\mb b =  \mb v_1 \, \| \mb b \|$.
   Then we have
   \begin{equation}
       \mb{r}_k(\delta) = \mb{r}_0(\delta) - (A-\delta I_n)V_k \mb{y}, \,\, \text{with} \,\, \mb{y} = \text{argmin}_\mb{y} \|\mb{r}_0(\delta) - (A-\delta I_n)V_k \mb{y}\|. \label{eq:res}
   \end{equation}
   Note that $\mb{r}_0 - (A-\delta I_n)V_k \mb{y} $ can be written as $p_k(A)\mb{r}_0(\delta)$ with $p_k$ a polynomial of degree at most $k$ and
   $p_k(\delta)=1$. Then $\mb{r}_k(\delta) = p_k(A)\mb{r}_0$, $p_k = \widetilde{p_k}/\widetilde{p_k}(\delta)$, with
   \begin{equation}
       \widetilde{p_k} =  \text{argmin}_{p_ \in \mathcal{P}_k} \frac{\|p(A)\mb{r}_0(\delta)\|}{|p(\delta)|} =
       \text{argmin}_{p \in \mathcal{P}_k} \frac{\|p\|_{A,\mb v_1}}{|p(\delta)|}, \label{eq:5}
   \end{equation}
   where $\mathcal{P}_k$ denotes the set of polynomials of degree at most $k$, and we use the norm induced by \eqref{eq:scalar_product}.
   Note that if $p(z) = \sum_{j=0}^k c_j q_j(z)$, then $\|p\|_{A,\mb v_1}^2 = \sum_{j=0}^k |c_k|^2$ by orthonormality. Therefore, the Cauchy-Schwarz inequality yields
   \begin{equation}
       \frac{\|p\|_{A,\mb v_1}}{|p(\delta)|} \geq \frac{1}{\sigma_k(\delta)}, \label{eq:6}
   \end{equation}
   where the minimum is attained for $c_j = \overline{q_j(\delta)}$. Combining \eqref{eq:6} and \eqref{eq:5} we conclude that \eqref{eq:respol} holds.
   In particular, $\| \mb r_k(\delta) \| = \| b \|/\sigma_k(\delta)$, which together with \eqref{eq:arnoldi} implies \eqref{eq:respolnorm}.
\end{proof}

\begin{remark}
As $\sigma_{m-1}(\delta)/\sigma_{N-\ell+1} = ||\mb r_{N-\ell+1}(\delta)||/||\mb r_{m-1}(\delta)||$, we see that the decay rates in Corollary~\ref{cor:semisep} and
Corollary~\ref{cor:semisep2} are linked to the convergence of the GMRES algorithm. If the GMRES algorithm
converges faster, the decay pattern becomes more pronounced.
\end{remark}

\begin{remark}\label{rem:GMRESrate}
   By taking norms in \eqref{eq:respol}, we see that for the relative GMRES residual
   \[
           \frac{\| \mb r_k(\delta)\|}{\| \mb r_0(\delta) \|} = \frac{1}{\sigma_k(\delta)} = \min_{p \in \mathcal P_k} \frac{\| p(A)\mb b\|}{\| b \| |p(\delta)|}.
   \]
   More generally, for the decay rates in Corollary~\ref{cor:semisep} and Corollary~\ref{cor:semisep2} for $\ell\leq k-m +2$, we have
   \begin{equation} \label{eq:GMRESrate1}
          \frac{\sigma_{m-1}(\delta)}{\sigma_{k-\ell+1}(\delta)}
          \leq \min_{p \in \mathcal P_{k+2-m-\ell}}
          \frac{\| p(A)\mb w_{m-1}(\delta)\|}{\| \mb w_{m-1}(\delta)\| \, |p(\delta)|},
   \end{equation}
   Let $\Omega\subset \mathbb C$ be a simply connected and compact $K$-spectral set for $A$; that is, $\| \pi(A) \| \leq K \, \max_{z\in \Omega} |\pi(z) |$
   for all polynomials $\pi$ (and hence $\Lambda(A)\subset \Omega$) with $\delta\not\in \Omega$. Also, let $\varphi$ be a map, mapping $\mathbb C \setminus \Omega$
   conformally onto the exterior of the closed unit disk. Then it can be shown that
   the right-hand side of \eqref{eq:GMRESrate1} can be bounded above by $| 1/\varphi(\delta)|^{k+2-\ell-m}<1$ times a modest constant,
   see, e.g., \cite[Chapter 6.11.2]{saad} for the case where $\Omega$ is an ellipse.
\end{remark}

\begin{example}\label{ex:unitary2} As in Example~\ref{ex:unitary1}, assume the matrix $A$ is unitary.
Given a polynomial $p$ of degree $k$, its reversed polynomial $p^\ast$ is defined as \[p^\ast(z) := z^k \overline{p(1/\overline{z}).}\]
The orthogonal polynomials $q_k(z)$ can be expressed as
\begin{eqnarray}
        q_k(z) &=& q_k^\ast(0) \, \text{argmin}_{q \,\text{monic of degree}\, k} \|q\|_{A,\mb v_1} \nonumber \\
        &=& q_k^\ast(0) \, \text{argmin}_{q \,\text{of degree}\, k} \frac{\|q\|_{A,\mb v_1}}{|q_k^\ast(0)|}. \label{eq:rev}
\end{eqnarray}
Note that $q_k^\ast(0)$ is the leading coefficient of $q_k$. As $A$ is unitary, $\|q\|_{A,\mb v_1} = \|q^\ast\|_{A,\mb v_1}$. Hence, \eqref{eq:rev} can be rewritten as
\begin{equation*}
   q_k(z) = q_k^\ast(0) \text{argmin}_{q \,\text{of degree}\, k} \frac{\|q^\ast\|_{A,\mb v_1}}{|q^\ast(0)|}
 \end{equation*}
Therefore,
 \begin{equation}
   q_k^\ast(z) = q_k^\ast(0) \text{argmin}_{q \in \mathcal{P}_k, q(0) \neq 0} \frac{\|q\|_{A,\mb v_1}}{|q(0)|}. \label{eq:reversepoly}
 \end{equation}
 Because of \eqref{eq:qdet}, $q_k^\ast(0) \geq 0$. Also, $\|q_k^\ast\|_{A,\mb v_1}=1$. Combining this with
 \eqref{eq:5} and \eqref{eq:respol}, \eqref{eq:reversepoly} yields
 \begin{equation}
      q_k^\ast(z) = \frac{1}{\sigma_k(0)} \sum_{j=0}^k \overline{q_j(0)}q_j(z), \label{eq:reversedeq}
 \end{equation}
 Therefore, the $k$th normalized GMRES residual $\mb{w}_k(0)$ of a unitary matrix $A$ can be expressed as
 \begin{equation}
     \mb{w}_k(0) = q_k^\ast(A)\mb{v}_1. \label{eq:unigmres}
 \end{equation}
\end{example}

\subsection{A progressive GMRES residual formula}
\label{sec:prog}
By \eqref{eq:respolnorm} we have that the $k$th normalized GMRES residual satisfies
\begin{eqnarray}
     \mb{w}_{k}(\delta) &=& \frac{1}{\sigma_k(\delta)}\sum_{i=0}^{k} \overline{q_i(\delta)}\mb{v}_{i+1}  \nonumber \\
              &=& \frac{\sigma_{k-1}(\delta)}{\sigma_{k}(\delta)} \frac{1}{\sigma_{k-1}(\delta)} \sum_{i=0}^{k-1} \overline{q_i(\delta)}\mb{v}_{i+1}
              + \frac{\overline{q_{k}(\delta)}}{\sigma_{k}(\delta)} \mb{v}_{k+1}  \nonumber \\
              &=& \frac{\sigma_{k-1}(\delta)}{\sigma_{k}(\delta)}\mb{w}_{k-1}(\delta) + \frac{\overline{q_{k}(\delta)}}{\sigma_{k}(\delta)} \mb{v}_{k+1},  \nonumber
             \\ &=&
             s_k(\delta) \mb{w}_{k-1}(\delta) + (-1)^k \overline{c_k(\delta)} \mb{v}_{k+1}, \label{eq:update}
\end{eqnarray}
the last equality following from \eqref{eq:cs}.
This demonstrates the existence of an updating formula to compute the residual vectors progressively. This formula can also be derived by means of
the $QR$-factorization of $\underline H_k - \delta \underline I_k$, see Proposition~\ref{prop:QR} and \cite[Subsection~6.5.3]{saad}.
The next result shows that it is not necessary to compute such a factorization for obtaining $s_k(\delta)$ and $c_{k}(\delta)$, if one is willing to compute the additional scalar product \eqref{eq:tau}.

\begin{proposition} \label{prop:tau}
    Define $\tau_k(\delta) = \mb e_k^\ast Q_k(\delta)^\ast (H_k-\delta I_k)\mb e_k $. Then
    \begin{equation}
       \tau_k(\delta) = \mb{w}_{k-1}(\delta)^\ast(A-\delta I_n)\mb{v}_k (-1)^{k-1},\label{eq:tau}
    \end{equation}
    and
\begin{equation}
 s_k(\delta) = \frac{h_{k+1,k}}{\sqrt{h_{k+1,k}^2 + |\tau_k(\delta)|^2}} \quad \text{and} \quad
 c_k(\delta) = \frac{\tau_k(\delta)}{\sqrt{h_{k+1,k}^2 + |\tau_k(\delta)|^2}}. \label{eq:update1}
\end{equation}
\end{proposition}
\begin{proof}
  From Proposition~\ref{prop:QR} and \eqref{eq:respolnorm} it follows that
 \begin{eqnarray}
V_k Q_k(\delta) \mb{e}_k &=& V_k \left(\begin{array}{c}
                             \overline{q_0(\delta)} \\
                             \vdots \\
                             \overline{q_{k-1}(\delta)}
                            \end{array}\right) \frac{(-1)^{k-1}}{\sigma_{k-1}(\delta)}\nonumber \\
                    &=& \frac{(-1)^{k-1}}{\sigma_{k-1}(\delta)} \sum_{i=1}^k \mb{v}_i \overline{q_{i-1}(\delta)}
                    %\nonumber \\
                    = (-1)^{k-1}\mb{w}_{k-1}(\delta) . \label{eq:20}
 \end{eqnarray}
 Therefore,
\begin{eqnarray*}
 \tau_k(\delta) &=& \mb{e}_k^\ast Q_k(\delta)^\ast (H_k-\delta I_k)\mb{e}_k \\
                &=& \mb{e}_k^\ast Q_k(\delta)^\ast V_k^\ast V_k (H_k-\delta I_k)\mb{e}_k \\
                &=& \mb{w}_{k-1}(\delta)^\ast (A\mb{v}_k - h_{k+1,k}\mb{v}_{k+1} - \delta \mb{v}_k)(-1)^{k-1} \\
                &=& \mb{w}_{k-1}(\delta)^\ast (A-\delta I_n)\mb{v}_k (-1)^{k-1},
\end{eqnarray*}
establishing \eqref{eq:tau}. We now claim that
\begin{equation}
 \mb{w}_k(\delta) \perp (A-\delta I_n)\mb{v}_j, \quad \text{for} \,\,\, 1 \leq j \leq k, \label{eq:perp}
\end{equation}
which is a direct consequence of \eqref{eq:res}:
\begin{equation*}
 \mb{r}_k(\delta) = \mb{r}_0(\delta) - \mb{u}_k \quad \text{with} \quad \mb{u}_k = \text{argmin}_{\mb{u} \in \mathcal{K}} \|\mb{r}_0(\delta) - \mb{u}\|,
\end{equation*}
where $\mathcal{K} = (A-\delta I_n)\text{span}\{\mb{v}_1, \ldots, \mb{v}_k \}$.

It remains to prove \eqref{eq:update1}.
Taking inner products with $(A-\delta I_n)\mb{v}_k$ in all terms of \eqref{eq:update} and making use of \eqref{eq:tau} and \eqref{eq:perp}, results in
\begin{equation}
 0 =  s_k(\delta) \overline{\tau_k(\delta)} (-1)^{k-1} + \overline{c_k(\delta)} h_{k+1,k} (-1)^k.
\end{equation}
Together with $s_k^2(\delta) + |c_k(\delta)|^2 = 1$, this yields \eqref{eq:update1}.
\end{proof}

\begin{example}\label{ex:unitary3}
   Let us return to the particular case of $\delta=0$ and unitary $A$ as discussed in Examples~\ref{ex:unitary1} and~\ref{ex:unitary2}.
   Inserting \eqref{eq:arnoldi} and \eqref{eq:unigmres} in \eqref{eq:update} and identifying the underlying polynomials gives
   \begin{equation} \label{eq:szegoe1}
        q_k^*(z)= s_k(0) q_{k-1}^*(z)  + (-1)^k \overline{c_k(0)} q_{k}(z).
   \end{equation}
   Since $q_{k-1}^*(z)-q_{k-1}^*(0)$ is $z$ times a polynomial of degree $<k-1$ and $q_0(z)=1$, we get from \eqref{eq:tau} that
   \begin{eqnarray*}
      \tau_k(0)&=&  (-1)^{k-1} \, \langle z q_{k-1}, q_{k-1}^* \rangle_{A,\mb v_1}
      \\&=&
      (-1)^{k-1} \, \overline{q_{k-1}^*(0)} \langle z q_{k-1}, q_0 \rangle_{A,\mb v_1}
           \\&=& (-1)^{k-1} \, \sigma_{k-1}(0) h_{1,k} = c_k(0)
   \end{eqnarray*}
   and $h_{k+1,k}=s_k(0)$, where we applied \eqref{eq:reversedeq} and
   \eqref{eq:some_entries}. Notice that this simplification for unitary $A$ is in accordance with \eqref{eq:update1}. Taking the star operation in \eqref{eq:szegoe1} gives the second relation
   \begin{equation} \label{eq:szegoe2}
        q_k(z)= s_k(0) z q_{k-1}(z)  + (-1)^k c_k(0) q_{k}^*(z).
   \end{equation}
   We should mention that, in case of unitary $A$, the scalar product \eqref{eq:scalar_product} can be written as a scalar product in terms of a (discrete)
   measure supported on the unit circle. Thus we have the whole theory  of orthogonal polynomials on the unit circle, and in particular relations
   \eqref{eq:szegoe1} and \eqref{eq:szegoe2} are known as the Szeg\H{o} recurrence relations, a coupled two-term recurrence for orthonormal polynomials on
   the unit circle \cite[Formula 1.2-1.7]{JaRe93}.
\end{example}

\section{The Arnoldi process for $BML$-matrices} \label{sec:rec}

We will establish a fast variant of the Arnoldi process which is applicable to the class of $BML$-matrices as described by formula \eqref{eq:introeq}.
In \S \ref{sec:hess} we will use the explicit representation of $(H_{k+1}-\delta I_{k+1})^{-*}$ in terms of orthogonal polynomials as stated in \eqref{eq:resolvent}
to demonstrate that the underlying Hessenberg matrix has
an upper-separable structure. The assumption on simple poles makes it possible to easily express generators for the upper-separable structure in polynomial language. In particular,
the GMRES residual vectors, which can be expressed in terms of orthogonal polynomials by \eqref{eq:respolnorm}, are showing up.
In \S \ref{sec:shortbml} we will see that the orthonormal basis vectors (up to a correction term incorporating the low rank perturbation in \eqref{eq:introeq}) can be written as a
linear combination of previously computed basis vectors and GMRES residual vectors.
In \S \ref{sec:pseudocode} the results from \S \ref{sec:shortbml} and \S \ref{sec:prog}, enabling to compute the necessary GMRES residual vectors progressively, will be combined into a new Arnoldi
iteration for $BML$-matrices.

\subsection{Structure formula for $BML$-matrices} \label{sec:hess}
%\comment{Some small changes are made below, e.g., some sentences are reformulated.}

 A rank structure revealing formula for $BML$-matrices will be derived in terms of the orthogonal polynomials $q_k$ defined in \S\ref{subsec:OP}.
 This formula \eqref{eq:hess} will be the key for the design of short recurrence relations in the following subsection. We first show that
 the $BML$-structure \eqref{eq:introeq} is inherited by the underlying Hessenberg matrix.

 \begin{proposition}\label{prop:heritage}
    Let $A$ be an invertible matrix satisfying relation \eqref{eq:introeq}, and denote by $N\leq n$ the Arnoldi termination index. Then $H_N$ is also a $BML$-matrix with the same
    polynomials $p,q$ and indices $\widetilde m_1 = m_1$, $\widetilde m_2 = m_2$ and $\widetilde m_3 \leq m_3$. More precisely,
    \begin{equation} \label{eq:hessstruct}
          H_N^\ast = p(H_N)q(H_N)^{-1} + F_N G_N^\ast
    \end{equation}
    where $F_k: = V_k^\ast F$ and $G_k := V_k^\ast G$.
 \end{proposition}
 \begin{proof}
    By definition of $N$ we have $AV_N = V_N H_N$, i.e., the columns of $V_N$ span an invariant subspace of $A$. By recurrence on the degree one shows that
    $\pi(A)V_N = V_N \pi(H_N)$, or $\pi(H_N) = V_N^\ast \pi(A) V_N$, for any polynomial $\pi$. Moreover, if $\pi(A)$ is of full rank, then so is $\pi(H_N)$.
    In particular, since we assume $q(A)$ to be invertible, so is the matrix $q(H_N)$. Hence, $V_N q(H_N)^{-1} = q(A)^{-1} V_N$. This implies that
    \begin{eqnarray*}
       H_N^\ast - p(H_N)q(H_N)^{-1}
       &=&
       V_N^\ast ( A^\ast - p(A)q(A)^{-1} ) V_N
       \\&=& V_N^\ast F G^\ast V_N,
    \end{eqnarray*}
    as claimed in \eqref{eq:hessstruct}.
 \end{proof}

A combination with Corollary~\ref{cor:semisep2} gives us the following result.

\begin{corollary}\label{cor:semisep3}
    Let $A$ be a matrix satisfying relation \eqref{eq:introeq}. Denote by $N\leq n$ the Arnoldi termination index, and $m:=\max(0,m_1-m_2+1)$. Then $H_N \in \mathbb{C}^{N \times N}$ is
    $(m,m_2+m_3)-$upper-separable. More precisely,
    \begin{align}
     \left((H_N)_{k,\ell}\right)_{\ell = k+m, \ldots, N} &= \sum_{j=1}^{m_2} \overline{d_j q_{k-1}(z_j)}
     \left( \left(( H_N - z_j I_N)^{-*}\right)_{1, \ell} \right)_{\ell = k+m, \ldots, N} \nonumber
     \\ &\phantom{=} + \left(\left(G_N F_N^\ast\right)_{k, \ell}\right)_{\ell = k+m, \ldots, N}, \label{eq:structformula}
    \end{align}
    where $z_1, \ldots, z_{m_2}$ denote the poles of the rational function $p(z)/q(z)$.
\end{corollary}
\begin{proof}
    The fact that $H_N$ is $(m,m_2+m_3)-$upper-separable is already known from \cite{mertensvandebril}. Below we give an alternative, more constructive, proof
    leading to explicit generators for the low rank part of $H_N$.
    By assumption, the rational function $p/q$ in \eqref{eq:introeq} has simple poles and thus has the partial fraction decomposition
    \begin{equation*}
        \frac{p(z)}{q(z)} = \sum_{j=1}^{m_2} \frac{d_j}{z-z_j} + \pi(z),
    \end{equation*}
    for some constants $z_j$, $d_j$ and a polynomial $\pi$ of degree $m_1-m_2$ ($\pi=0$ if $m_1 < m_2$). Replacing $z$ by $H_N$, taking
    adjoints and using \eqref{eq:hessstruct} and \eqref{eq:resolvent} leads to
    \begin{eqnarray} && H_N - G_N F_N^* - \pi(H_N)^*
    \nonumber
    \\&=& \sum_{j=1}^{m_2} \overline{d_j} ( H_N - z_j I_N)^{-*}
    \nonumber
    \\&=& L_N + \sum_{j=1}^{m_2} \overline{d_j}
    \Bigl(q_0(z_j),\ldots,q_{N-1}(z_j)\Bigr)^\ast \mb e_1^\ast ( H_N - z_j I_N)^{-*},
    \label{eq:hess}
    \end{eqnarray}
    with a strictly lower triangular matrix $L_N$.
    Finally, according to the upper Hessenberg structure of $H_N$, the matrix $\pi(H_N)^*$ has zero entries on and above the $m$th diagonal, establishing the
    upper-separable structure and formula \eqref{eq:structformula}.
 \end{proof}

 %\comment{Remark added}
\begin{remark}
The assumption on simple poles is not necessary for $H_N$ to have an upper-separable structure.  However, once we drop this constraint, it
is not
clear whether or not there exists a link between the generators of the low-rank structure and the GMRES residual vectors. We refer to \cite{mertensvandebril}
for more information on this topic.

Note that $H_k$ for $k< N$ also has an upper-separable structure as it is a leading principal minor (submatrix) of $H_N$.
However, $H_k$ does not satisfy the same matrix equation as $H_N$.
\end{remark}
% \begin{remark}
%  A similar rank structure formula as \eqref{eq:hess} can be derived when dropping the restriction \eqref{eq:str2} of simple poles.
%  Corollary \ref{cor:semisep} can be extended, determining the structure of the $Q$-factor in the $QR$-decomposition of $q(H_N)$, as a matrix of
%  rank $\text{deg}(q)$ plus a lower triangular matrix with zeros above the $\text{deg}(q)$th subdiagonal. Incorporating this into \eqref{eq:hessstruct}
%  and using the same reasoning as above, reveals the rank structure of $H_N$ as stated in \cite{mertensvandebril}.
% \end{remark}

\subsection{Short recurrence relations for $BML$-matrices} \label{sec:shortbml}
%\comment{Raf didn't like the structure of this section. It was not clear to him what we where trying to achieve with the proposition below. For example, what is the need of this $\mb{v}_k'$? So I
%wrote some introduction in which I tried to give some more background, explain our goals and then introduce the proposition below. I also incorporated Example 3.5 of the previous version
%into this small introduction. In the proposition below I included a recurrence formula, which we can use to refer to later on. Also, there are some changes made to the proof. I elaborated a
%little bit more on some of the details, hopefully improving the readability.}

We will now derive a short recurrence relation for $BML$-matrices. To do so, the vector
 \begin{equation} \label{eq:defvk}
          \mb v'_k := A \mb v_k - V_{k-m} G_{k-m} F^* \mb v_{k},
    \end{equation}
   is introduced, which is equal to $A\mb v_k$ up to a correction term induced by the low rank perturbation in \eqref{eq:introeq}. In \cite{BeckReich} it is
   shown that for nearly Hermitian $A$, i.e., $p(z) = z$ and $q(z) = 1$ in \eqref{eq:introeq}, the Arnoldi vectors satisfy a three term recurrence relation.
   More specifically,
 \begin{equation}
  \mb{v}_k' = h_{k-1,k} \mb{v}_{k-1} + h_{k,k} \mb{v}_k + h_{k+1} \mb{v}_{k+1,k}, \label{eq:nearlyherm}
 \end{equation}
where $h_{k-1,k} = \mb{v}_{k-1}^\ast \mb{v}_k'$, $h_{k,k} = \mb{v}_{k}^\ast \mb{v}_k'$ and $h_{k+1,k} = \mb{v}_{k+1}^\ast \mb{v}_k'$ are entries of the underlying Hessenberg
matrix. We refer to \cite{BeckReich} for a detailed discussion.
   For a general $BML$-matrix, Proposition \ref{prop:short_rec} states that $\mb v_k'$ is a linear combination of GMRES residual vectors and Arnoldi vectors, including
   $\mb{v}_{k+1}$. As we will discuss below, this results in a short recurrence relation which reduces to \eqref{eq:nearlyherm} in the specific case of nearly
   Hermitian matrices and which can be used to compute the Arnoldi vectors in an efficient way.

\begin{proposition}\label{prop:short_rec}
    Let $A$ be a matrix satisfying relation \eqref{eq:introeq}, and denote by $N\leq n$ the Arnoldi termination index, and $m:=\max(0,m_1-m_2+1)$.
    Then for $m < k \leq N$, the vector $\mb v_k'$ as defined by \eqref{eq:defvk}
    can be written as a linear combination of $\mb v_j$ for $j=k-m+1,...,k+1$ and of $\mb w_{k-m-1}(z_j)$ for $j=1,...,m_2$. More precisely,
    \begin{equation} \mb{v}'_k = \sum_{j=1}^{m_2} a_{j,k} \mb{w}_{k-m-1}(z_j) + \sum_{j=k-m+1}^{k+1} h_{j,k} \mb{v}_j, \label{eq:8} \end{equation}
    for some constants $a_{j,k}$, $h_{j,k}$ entries of the underlying Hessenberg matrix $H_N$, and $m < k \leq N$.
\end{proposition}
\begin{proof}
    Notice that by construction,
    \[
       \mb v''_k  := \mb v'_k  - \sum_{j=k-m+1}^{k+1} h_{j,k} \mb v_j
    \]
    lies in the Krylov space spanned by the columns of $V_{k-m}$. As a result we have
    \[
    \mb v''_k=V_{k-m}V_{k-m}^\ast \mb v''_k=V_{k-m}V_{k-m}^\ast \mb v'_k.
    \]
    Define
    \[
    \widehat I_{k-m} := \left[\begin{array}{c} I_{k-m}\\ \mb{0}\\ \vdots\\ \mb{0} \end{array} \right] \in \mathbb{C}^{N \times (k-m)}.
    \]
    Then $V_{k-m} = V_N \widehat I_{k-m}$. Combining the above yields
    \begin{eqnarray*}
          \mb v'_k &=& \mb{v}_k' - \mb{v}_k'' + V_{k-m}V_{k-m}^\ast \mb v'_k
          \\&=&
          \mb v_k' - \mb v_k'' + V_{k-m} \left( V_{k-m}^* A V_{N} - V_{k-m}^* G F_N^* \right) \mb e_k
          \\&=&
          \mb v_k' - \mb v_k'' + V_{k-m}  \widehat I_{k-m}^\ast \left(H_N-G_N F_N^* \right) \mb e_k.
    \end{eqnarray*}
   The vector $\widehat I_{k-m}^\ast \left(H_N-G_N F_N^*\right) \mb e_k \in \mathbb{C}^{k-m}$ consists of the first $k-m$ components of the $k$th column of $H_N-G_N F_N^*$.
    As the entries of $L_N$ and $\pi(H_N)^\ast$ in \eqref{eq:hess} are zero on and above the $m$th diagonal, it follows that
    \begin{equation*}
     \widehat I_{k-m}^\ast \left(H_N-G_N F_N^*\right) \mb e_k = \sum_{j=1}^{m_2} \overline{d_j} \left( q_0(z_j),...,q_{k-m-1}(z_j) \right)^\ast \mb{e}_1(H_N -z_j I_N)^{-\ast}\mb{e}_k,
    \end{equation*}
    which by \eqref{eq:respolnorm} can be rewritten as
    \begin{equation*}
     V_{k-m}\widehat I_{k-m}^\ast \left(H_N-G_N F_N^*\right) \mb e_k = \sum_{j=1}^{m_2} a_{j,k} \mb w_{k-m-1}(z_j),
    \end{equation*}
    with $a_{j,k} := \overline{d_j} \sigma_{k-m-1}(z_j) \mb{e}_1(H_N -z_j I_N)^{-\ast}\mb{e}_k$.
As a result, $\mb v'_k$ can be written as $\mb v_k' - \mb v_k''$, a linear combination of
    $\mb v_{k+1},...,\mb v_{k-m+1}$ (with coefficients being entries of the underlying Hessenberg matrix), plus a linear combination
    of $\mb w_{k-m-1}(z_1),...,\mb w_{k-m-1}(z_m)$.
\end{proof}

%\comment{It made the statements in the remark below a little bit more exact.}
\begin{remark}\label{rem:short_rec}
   Proposition~\ref{prop:short_rec} remains true if we replace  $\mb w_{k-m-1}(z_j)$ by $\mb w_{\ell}(z_j)$ or/and $V_{k-m}G_{k-m}$ by $V_{\ell+1}G_{\ell+1}$ for
   $\ell  \in \{k-m-1,k-m,...,k\}$. This is the direct result of the observation that both $V_{k-m}G_{k-m}F^\ast \mb{v}_k - V_{\ell+1}G_{\ell+1} F^\ast \mb{v}_k $ and
   $\mb{w}_{k-m-1}(z_j) - \sigma_\ell(z_j)/\sigma_{k-m-1}(z_j)\mb{w}_\ell(z_j) $ are elements of $\mbox{span}(\mb v_{k-m+1}, \ldots, \mb{v}_{\ell+1})$ if $\ell > k-m-1$
   and equal to zero if $\ell = k-m-1$.
   However, choosing $\ell \neq k-m-1$ causes to lose orthogonality between $\mb v_k' - \mb v_k''$ and $\mb v_k''$ and to loose the link with the entries
   of the underlying Hessenberg matrix in \eqref{eq:8}.
\end{remark}

% The following corollary illustrates how the decay property on the entries of a Hessenberg matrix as discussed in \S\ref{subsec:QR} is reflected in the coefficients of
% the recurrence relation \eqref{eq:8}.
%
% \begin{corollary}
% We have that
%  \[|a_{j,k}| \leq ||(H_N-z_j I_N)^{-1}||\, |d_j| \sigma_{k-m-1}(z_j)/\sigma_{k-1}(z_j). \]
% \end{corollary}
% \begin{proof}
%  From Corollary \ref{cor:semisep2} it follows that
%  \[
%   |\mb{e}_1 (H_N - z_j I_N)^{-\ast} \mb{e}_k| \leq ||(H_N-z_j I_N)^{-1}||/\sigma_{k-1}(z_j).
%  \]
%  Combining with \eqref{eq:hess} and \eqref{eq:respolnorm} yields
%  \begin{eqnarray*}
%   |a_{j,k}| &\leq& ||(H_N-z_j I_N)^{-1}||\, |d_j| \sigma_{k-m-1}(z_j)/\sigma_{k-1}(z_j),
%  \end{eqnarray*}
% proving the above statement.
% \end{proof}
%
% As $\sigma_{k-m-1}(z_j)/\sigma_{k-1}(z_j) = ||\mb{r}_{k-1}(z_j)||/||\mb{r}_{k-m-1}(z_j)||$, we see that the norm of $a_{j,k}$ is linked to the speed of convergence of
% the GMRES algorithm. The more rapid the GMRES residual vectors decrease in norm, the smaller the coefficients $a_{j,k}$ become.

Next, let us show how the recurrence relation of Proposition~\ref{prop:short_rec} reduces to the well-known Szeg\H{o} recurrence if the matrix under consideration is unitary.

\begin{example}\label{ex:unitary4}
   Let us return to the particular case of $\delta=0$ and unitary $A$ as discussed in Examples~\ref{ex:unitary1},~\ref{ex:unitary2}, and~\ref{ex:unitary3}.
   Inserting \eqref{eq:arnoldi} and \eqref{eq:unigmres} in the second Szeg\H{o} relation \eqref{eq:szegoe2}
   leads to
   \[
        \mb v_{k+1}= s_k(0) A \mb v_{k}  + (-1)^k c_k(0) \mb w_{k}(0)
   \]
   which is exactly the variant with $\ell=k, m_2=m_3=1, m_1=m=0$ of Proposition~\ref{prop:short_rec} discussed in Remark~\ref{rem:short_rec}.
\end{example}

%\comment{Remark reformulated}
\begin{remark} \label{remark:wellfree}
 In \cite{bella} a class of matrices, called $(H,m)$-well-free matrices is investigated, and it is established that these matrices satisfy the recurrence relation \eqref{eq:multi} with $m_1=m_2, m_3=0$.
 Briefly described, these matrices form a subset of
 upper-separable Hessenberg matrices which satisfy an additional constraint preventing a breakdown in the recurrence relation \eqref{eq:multi}.
 Intuitively, this additional ``well-free'' constraint signifies that
 there are no rank deficiencies encountered in the low rank part of the upper-separable Hessenberg matrix. The problem of breakdown is discussed in \cite{barthmanteuffel} and
 \cite{mertensvandebril}, where it is overcome by making
 use of a set of several multiple recurrence relations instead of the single recurrence relation \eqref{eq:multi}, and the use of an algorithm based on \eqref{eq:multi} which provides a set of column
 vectors to generate the low rank structure
 of the underlying Hessenberg matrix, respectively. We will, however, not need
 any well-free constraint to prevent a breakdown in the recurrence relation stated in Proposition~\ref{prop:short_rec}, as our approach does not impose any
 limitations on the matrix structure beyond \eqref{eq:introeq}.

\end{remark}
%\comment{Corollary added}

%\comment{Corollary added}

\subsection{The algorithm} \label{sec:pseudocode}
%\comment{New section}
Algorithm \ref{alg:bml}, which we will name \emph{Fast Arnoldi} throughout the text,  describes a fast variant of the Arnoldi algorithm for $BML$-matrices based on Proposition \ref{prop:short_rec}. We will give a short
description of each of the components of the algorithm and print the corresponding piece of pseudocode.

The idea is to make alternate use of the recurrence relations
\begin{eqnarray}
 \mb{v}'_k &=& \sum_{j=1}^{m_2} a_{j,k} \mb{w}_{k-m-1}(z_j) + \sum_{j=k-m+1}^{k+1} h_{j,k} \mb{v}_j, \label{eq:computenextv} \\
 \mb w_k(z_j) &=& s_k(z_j) \mb{w}_{k-1}(z_j) + (-1)^k \overline{c_k(z_j)} \mb{v}_{k+1}. \label{eq:computenextw}
\end{eqnarray}

The recurrence relation \eqref{eq:computenextv} is used to compute the next orthonormal basis vector of the Krylov subspace as a linear combination of previously
computed orthonormal vectors as well as GMRES residual vectors, while the recurrence relation \eqref{eq:computenextw} is
used to update the GMRES residual vectors once a new orthonormal vector is retrieved. As \eqref{eq:computenextv} is only valid for $k > m$, the first $m$ orthonormal
vectors are computed by means of the classical Arnoldi iteration.

Each time the relation \eqref{eq:computenextv} is employed, the vector $\mb{v}_k'$ is formed, causing products between vectors and matrices to be computed. The total complexity
to compute $\mb{v}_k'$ is $\mathcal{O}(m_3 n) + \mathcal{O}(n^2)$. \newline
\begin{tabular}{c}
\begin{algorithm*}[H]
 $\hat F_{k} := \mb{v}_k^\ast F$; $\hat G_{k-m} := \mb{v}_{k-m}^\ast G$\;
  $\widetilde G := \widetilde G + \mb{v}_{k-m} \hat G_{k-m}$\;
  $\mb{v}' := A \mb{v}_k - \widetilde G \hat F^\ast_{k}$\;
\end{algorithm*}
\end{tabular}

The coefficients $a_{j,k}$ and $h_{j,k}$ in \eqref{eq:computenextv} are the solution of the least squares problem
\begin{equation*}
 \mb{v}'_k = \left[\mb{w}_{k-m-1}(z_1), \ldots, \mb{w}_{k-m-1}(z_{m_2}), \mb{v}_{k-m+1}, \ldots, \mb{v}_{k+1}\right] \left[ \begin{array}{c}
                                                                                 a_{1,k}\\
                                                                                 \vdots \\
                                                                                 a_{m_2,k} \\
                                                                                 h_{k-m+1,k} \\
                                                                                 \vdots \\
                                                                                 h_{k+1,k}
                                                                                \end{array}\right].
\end{equation*}
Note that $\mb{v}_i \perp \text{span}\{\mb{w}_{k-m-1}(z_1), \ldots, \mb{w}_{k-m-1}(z_{m_2}) \}$ for all $k-m+1 \leq i \leq k+1$, allowing to solve the above least
squares problem without knowing $\mb{v}_{k+1}$ in advance. To shorten notation in subsequent discussions, we define
$M_k := \left[ \mb{w}_{k-m-1}(z_1), \ldots, \mb{w}_{k-m-1}(z_{m_2}) \right]$.

Each of the coefficients $h_{j,k}$ are entries of the $k$th column of the corresponding Hessenberg matrix and are computed as $\mb{v}_j^\ast \mb{v}_k'$, which has a
computational complexity of $\mathcal{O}(mn)$. \newline
\begin{tabular}{c}
  \begin{algorithm*}[H]
 \For{$j= k, k-1, \ldots, k-m+1$}{
   $h_{j,k} := \mb{v}_j^\ast \mb{v}'$; $\mb{v}' := \mb{v}' - h_{j,k} \mb{v}_j$\;
  }
\end{algorithm*}
\end{tabular}

Next, a $QR$-decomposition of the matrix $M_k$ is computed, after which the coefficients $a_{1,k}, \ldots a_{m_2,k}$
are retrieved by back-substitution. The complexity of this operation is $\mathcal{O}\left(m_2^2 (n+1)\right)$. \newline
\begin{tabular}{c}
\begin{algorithm*}[H]
 $Q := [\mb{q}_1, \ldots, \mb{q}_{m_2}], R := (r_{i,j}) \in \mathbb{C}^{m_2 \times m_2}$,\\ such that $QR = M_k$\;
  \For{$j=m_2, m_2-1, \ldots, 1$}{
  $a_{j,k} := \left(\mb{q}_j^\ast \mb{v}' - \sum_{\ell=j+1}^{m_2} r_{j,\ell} a_{j,k}\right)/r_{j,j}$\;
  $\mb{v}' = \mb{v}' - a_{j,k} \mb{w}_{k-m-1}(z_j)$\;
  }
  $h_{k+1,k} = ||\mb{v}'||$; $\mb{v}_{k+1} := \mb{v}'/h_{k+1,k}$\;
\end{algorithm*}
\end{tabular}
Then for each $z_i$, $1 \leq i \leq m_2$, recurrence relation \eqref{eq:computenextw} is used to update the GMRES residual vectors, which are all equal to
the starting vector $\mb{v}_1$ at the beginning of the iteration ($k=m+1$). Each time the relation \eqref{eq:computenextw} is employed, a matrix vector product needs to be computed.
This leads to a total complexity of $\mathcal{O}(m_2 n^2)$. \newline
\begin{tabular}{c}
\begin{algorithm*}[H]
  \For{$j=1, \ldots, m_2$}{
  $\tau_{k-m}(z_j) = (-1)^{k-m-1}\mb{w}_{k-m-1}(z_j)^\ast (A-z_j I_n) \mb{v}_{k-m}$\;
  $s_{k-m}(z_j) = h_{k-m+1,k-m}/\sqrt{h_{k-m+1,k-m}^2 + |\tau_{k-m}(z_j)|^2}$\;
  $c_{k-m}(z_j) = \tau_{k-m}(z_j)/\sqrt{h_{k-m+1,k-m}^2 + |\tau_{k-m}(z_j)|^2}$\;
  $\mb{w}_{k-m}(z_j) = s_{k-m}(z_j)\mb{w}_{k-m-1}(z_j) + (-1)^{k-m} \overline{c_{k-m}(z_j)} \mb{v}_{k-m+1}$\;
  }
\end{algorithm*}
\end{tabular}
Note that it is numerically more stable if we do not divide by the square root \newline $\sqrt{h_{k-m+1,k-m}^2 + |\tau_{k-m}(z_j)|^2}$ in the computation of $s_{k-m}(z_j)$ and $c_{k-m}(z_j)$, but
instead normalize $\mb{w}_{k-m}(z_j)$ after each iteration (this however, leads to $m_2$ additional scalar products).
If we assume the matrix under consideration is sparse; allowing a computational complexity of $\mathcal{O}(n)$ to compute a matrix vector product; the total complexity to compute the
first $k$ orthonormal Arnoldi vectors can be estimated as $\mathcal{O}(kn)$.

\begin{remark}
 If the rational function in \eqref{eq:introeq} has only one pole, i.e., $m_2 = 1$ in \eqref{eq:8} then the order in which the coefficients are determined can be reversed.
 More precisely, we can first compute $a_{1,k}$ as $\mb w_{k-m-1}(z_1)^\ast \mb v_k'$ and then orthonormalize the resulting difference $\mb v_k' - a_{1,k} \mb w_{k-m-1}(z_1)$
 against $\mb v_{k-m+1}, \ldots, \mb v_k$ to obtain $\mb v_{k+1}$. This might be of influence on the numerical performance.
\end{remark}

\begin{algorithm}[H]
 \KwData{$A \in \mathbb{C}^{n\times n}$ with $A^\ast = \sum_{j=1}^{m_2} d_j(A-z_j I_n)^{-1} + \pi(A)$, $\pi$ of degree $m_1-m_2$, $F=[\mb{f}_1, \ldots, \mb{f}_{m_3}]$,
 $G = [\mb{g}_1, \ldots, \mb{g}_{m_3}] \in \mathbb{C}^{n \times m_3}$, $\mb{b} \in \mathbb{C}^{n}$, $i$}
 \KwResult{$V_{i+1} = [\mb{v}_1, \mb{v}_2, \ldots, \mb{v}_{i+1}] \in \mathbb{C}^{n \times (i+1)}$}
  $m := \max (0, m_1-m_2+1) $\;

 $\mb{v}_1 = \mb{b}/||\mb{b}||$\;
 \For{$j=1, \ldots, m_2$}{
 $\mb{w}_0(z_j) = \mb{v}_1$\;
 }
 \For{$k=1, \ldots, m$}{
 $\mb{v}' = A\mb{v}_k$\;
 \For{$j=1, \ldots, i$}{
 $h_{j,k} = \mb{v}_j^\ast \mb{v}'$; $\mb{v}' = \mb{v}' - h_{j,k}\mb{v}_j$\;
 }
 $h_{k+1,k} = ||\mb{v}'||$; $\mb{v}_{k+1} = \mb{v}'/h_{k+1,k}$\;
 }
 $\widetilde G := 0 \in \mathbb{C}^{n \times m_3}$\;
  \For{$k=m+1, \ldots, i$}{
  $\hat F_{k,:} := \mb{v}_k^\ast F$; $\hat G_{k-m,:} := \mb{v}_{k-m}^\ast G$\;
  $\widetilde G := \widetilde G + \mb{v}_{k-m} \hat G_{k-m,:}$\;
  $\mb{v}' := A \mb{v}_k - \widetilde G \hat F^\ast_{k,:}$\;
  \For{$j= k, k-1, \ldots, k-m+1$}{
  $h_{j,k} := \mb{v}_j^\ast \mb{v}'$; $\mb{v}' := \mb{v}' - h_{j,k} \mb{v}_j$\;
  }
  $Q := [\mb{q}_1, \ldots, \mb{q}_{m_2}], R := (r_{i,j}) \in \mathbb{C}^{m_2 \times m_2}$, such that $QR = [\mb{w}_{k-m-1}(z_1), \ldots, \mb{w}_{k-m-1}(z_{m_2})]$\;
  \For{$j=m_2, m_2-1, \ldots, 1$}{
  $a_{j,k} := \left(\mb{q}_j^\ast \mb{v}' - \sum_{\ell=j+1}^{m_2} r_{j,\ell} a_{j,k}\right)/r_{j,j}$\;
  $\mb{v}' = \mb{v}' - a_{j,k} \mb{w}_{k-m-1}(z_j)$\;
  }
  $h_{k+1,k} = ||\mb{v}'||$; $\mb{v}_{k+1} := \mb{v}'/h_{k+1,k}$\;
  \For{$j=1, \ldots, m_2$}{
  $\tau_{k-m}(z_j) = (-1)^{k-m-1}\mb{w}_{k-m-1}(z_j)^\ast (A-z_j I_n) \mb{v}_{k-m}$\;
  $s_{k-m}(z_j) = h_{k-m+1,k-m}/\sqrt{h_{k-m+1,k-m}^2 + |\tau_{k-m}(z_j)|^2}$\;
  $c_{k-m}(z_j) = \tau_{k-m}(z_j)/\sqrt{h_{k-m+1,k-m}^2 + |\tau_{k-m}(z_j)|^2}$\;
  $\mb{w}_{k-m}(z_j) = s_{k-m}(z_j)\mb{w}_{k-m-1}(z_j) + (-1)^{k-m} \overline{c_{k-m}(z_j)} \mb{v}_{k-m+1}$ \;
  }}
 \caption{The fast Arnoldi method}
 \label{alg:bml}
\end{algorithm}

\section{Connection with the Barth-Manteuffel multiple recurrence relation} \label{sec:barthmanteuffel}

The aim of this section is to show how our work is related to that of Barth \& Manteuffel in their article on `Multiple recursion conjugate gradient algorithms'
\cite{barthmanteuffel}.
They introduce an economical conjugate gradient algorithm for the class of $BML$-matrices, by making use of short recurrence relations.
We will give a short summary of their findings and discuss both the differences and similarities with our approach.

To prevent the possibility of a breakdown in their so-called `single recurrence relation' Barth \& Manteuffel rewrote it as a set of recurrence relations, that are
stated in \eqref{eq:barth1}-\eqref{eq:barth3}.

\begin{align}
 \underline{\mb p}_{j+1} &= A \underline{\mb p}_j - \sum_{i=j-(\ell-m)}^j t_{i,j} \underline{ \mb p}_i
 - [\underline{\mb q}_{j_0}, \ldots, \underline{\mb q}_{j_{m-1}}]\underline{\eta}_j \nonumber
 \\ &\phantom{=}- [\underline{\hat{\mb q}}_{j_0}, \ldots, \underline{\hat{\mb q}}_{j_{\kappa-1}}] \underline{\mu}_j, \label{eq:barth1}\\
 \underline{\mb q}_{{j+1}_i} &= \frac{(\overline{\underline{\rho}}_{j+1})_i}{\underline{\mb p}_{j+1}^\ast \underline{\mb p}_{j+1}} \underline{\mb p}_{j+1} + \underline{\mb q}_{j_i} \qquad
 \text{for} \qquad i=0, \ldots, m-1, \label{eq:barth2}\\
 \underline{\hat{\mb q}}_{{j+1}_i} &= \frac{(\overline{\underline{\tau}}_{j+1})_i}{\underline{\mb p}_{j+1}^\ast \underline{\mb p}_{j+1}} \underline{\mb p}_{j+1} + \underline{\hat{\mb q}}_{j_i} \qquad
 \text{for} \qquad i=0, \ldots, \kappa-1. \label{eq:barth3}
\end{align}
Unfortunately, they use different letters, shift indices, and construct an orthogonal, but not orthonormal basis of the Krylov space.
The new normalization comes from the fact that they consider a Hessenberg matrix which has ones on its first subdiagonal.
Their basis vectors $\mb{\underline p}_0, \mb{\underline p}_1, \ldots, \mb{\underline p}_{N-1} \in \mathbb C^n$
satisfy $\mb{\underline p}_j/\|\mb{ \underline p}_j\|=\mb v_{j+1}$. Moreover, they use the integers $(\ell,m,\kappa,\theta)$
instead of $(m_1,m_2,m_3,m_3-1+m)$. Also, as seen in \eqref{eq:barth1}-\eqref{eq:barth3}, two other families of vectors with double indices are used. For simplicity and
consistency we will abbreviate them as \[W_{k}:=[\mb{\underline q}_{k_0},...,\mb{\underline q}_{k_{m_2-1}}]\in \mathbb C^{n \times m_2} \qquad \text{and}  \qquad \widehat W_{k} :=
[\mb{\underline{\widehat q}}_{k_0},...,\mb{\underline{\widehat q}}_{k_{m_3-1}}]\in \mathbb C^{n \times m_3}.\]

As the reader will see below, to compare our approach with \cite{barthmanteuffel}, we will not explicitly make use of \eqref{eq:barth1}-\eqref{eq:barth3}, but instead make use of a mathematical equivalent
of \eqref{eq:barth1}-\eqref{eq:barth3} which is adapted to our notation and scalings. The original pseudocode used by Barth \& Manteuffel is stated in Algorithm \ref{alg:pseudobarth}.
% Therefore, the reader should be aware that the statements below, although we are refering to
% \cite{BaMa00},
% are not an exact copy of the original work of Barth \& Manteuffel, as we are using different scalings.
% Finally, the following summary of their work can therefore exhibit slightly different numerical behavior on a computer.

In \cite[Eqn.\ (4.16)]{barthmanteuffel} the authors provide an explicit formula for the entries of the upper Hessenberg matrix $H$:
\begin{equation} \label{eq:BM_generator}
    H_{j,k} = \mb v_j^* A \mb v_k = \underline \rho_j^* \underline \eta_k
    + \underline \tau_j^* \underline \mu_k, \,\, \underline \rho_j, \underline \eta_k \in \mathbb C^{m_2} , \,\, \underline \mu_j , \underline \tau_k \in \mathbb C^{m_3},
\end{equation}
for all $j=1,2,...,k-m$,
in which the reader recognizes generators $\underline \rho_i, \underline \eta_i,  \underline \mu_i$ and  $\underline \tau_i$ for the low-rank part of $H$. Note that, in contrast to
\cite{barthmanteuffel}, we start numbering with $i=1$ instead of $i=0$. To be able to use the recurrence relations \eqref{eq:barth1}-\eqref{eq:barth3} in practice, the
generators $\underline \rho_i, \underline \eta_i,  \underline \mu_i$ and  $\underline \tau_i$ need to be known in advance. Therefore, Algorithm \ref{alg:pseudobarth}
is based on a rewritten form of the recurrence relations \eqref{eq:barth1}-\eqref{eq:barth3} which enables to compute the orthogonal basis vectors
$\mb p_i$ and the generators $\underline \rho_i, \underline \eta_i, \underline \mu_i$ and $\underline \tau_i$ simultaneously.

We define $W_k$ and $\widehat W_k$ \cite[Eqn.\ (4.22) and Eqn.\ (4.23)]{barthmanteuffel} as
\begin{equation} \label{eq:BM_qvectors}
    W_k = \sum_{j=1}^{k+1} \mb v_j \underline \rho_j^* \in \mathbb C^{n \times m_2}, \quad
    \widehat W_k = \sum_{j=1}^{k+1} \mb v_j \underline \tau_j^* \in \mathbb C^{n \times m_3},
\end{equation}
for $k = 0, \ldots, n-1$, which is the equivalent of \eqref{eq:barth2}-\eqref{eq:barth3}.
With a similar reasoning as in \cite[Eqn.\ (5.3)]{barthmanteuffel} it can be proven that
\begin{equation} \label{eq:BM_recurrence}
    A \mb v_k = \sum_{j=1}^{k+1} \mb v_j H_{j,k} = \sum_{j=k-m+1}^{k+1} \mb v_j H_{j,k}
    + W_{k-m-1} \underline \eta_k + \widehat W_{k-m-1} \underline \mu_k,
\end{equation}
which is mathematically equivalent to \eqref{eq:barth1}.

Suppose now that the generators as defined in \eqref{eq:BM_generator} are known. Then one can use
\eqref{eq:BM_recurrence} to compute $\mb v_{k+1}$ out of $\mb v_{k-m+1},...,\mb v_k, W_{k-m-1}$ and $\widehat W_{k-m-1}$, then use
\eqref{eq:BM_qvectors} to compute $W_{k-m},\widehat W_{k-m}$ out of $\mb v_{k-m+1}, W_{k-m-1}$ and $\widehat W_{k-m-1}$, and so on. Hence, it remains
to find the generators. Two of them can be computed with an explicit formula \cite[Eqn.\ (4.15) and (4.13)]{barthmanteuffel}, namely
\begin{equation} \label{eq:BM_mu_eta}
    \underline \mu_k = F^* \mb v_k , \quad
    \underline \eta_k = ( H_{j,k} )_{j=1,...,m_2}.
\end{equation}
The vector $\underline \rho_j$ is obtained \cite[Eqn.\ (4.12)]{barthmanteuffel} as the `remainder' in the polynomial division of $q_{j-1}$ \eqref{eq:arnoldi} by
the denominator $q$ \eqref{eq:introeq}:
\begin{equation} \label{eq:BM_rho}
        q_{j-1}(z) = \alpha_j(z) q(z) + (q_0(z),...,q_{m_2-1}(z) ) \underline \rho_j,
\end{equation}
\noindent where we observe that $\rho_1,...,\rho_{m_2}$ form the canonical basis of $\mathbb C^{m_2}$, and thus $V_{m_2}^* W_{k} = I$.
We may rewrite the remainder in terms of the Lagrange polynomials $\ell_h$ of the roots $z_1,..,z_{m_2}$ of $q$, leading to $$
     (q_0(z),...,q_{m_2-1}(z) ) \underline \rho_j =
     \sum_{h=1}^{m_2} q_{j-1}(z_h) \ell_h(z) \quad
     \mbox{or} \quad
     \underline \rho_j = \sum_{h=1}^{m_2} q_{j-1}(z_h) V_{m_2}^* \ell_h(A) \mb v_1 .
$$
Substituting the expression for $\underline \rho_j$ into \eqref{eq:BM_qvectors} allows to conclude that
\begin{equation}
 W_k = M_k \left[\begin{array}{c}
                                                      \sigma_k(z_1) \mb v_1^\ast \ell_1(A)^\ast V_{m_2} \\
                                                      \vdots \\
                                                      \sigma_k(z_{m_2}) \mb v_1^\ast \ell_{m_2}(A)^\ast V_{m_2}
                                                     \end{array}
 \right].
\end{equation}
Recalling that $V_{m_2}^* W_{k} = I$ gives
\begin{equation} \label{eq:link_w}
        W_k = M_k
        \Bigl( V_{m_2}^* M_k \Bigr)^{-1} ,
\end{equation}
which makes a partial link between \eqref{eq:BM_recurrence} and Proposition~\ref{prop:short_rec}. In particular, if the matrix $A$ is unitary and no
low-rank perturbation is involved, $W_k$ is a multiple of $\mb w_k(0)$ and the Barth-Manteuffel multiple recurrence relation turns out to be equivalent to the Szeg\H{o} recurrence
relations.
The quantities $\underline \rho_{k+1}$ are
computed recursively \cite[Eqn.\ (5.11)]{barthmanteuffel} by computing all entries of $\underline H_k$, and by taking remainders after division by $q$ in the relation $zq_{k-1}(z) =
\sum_{j=1}^{k+1} H_{j,k} q_{j-1}(z)$, the polynomial translation of the Arnoldi relation $AV_k = V_{k+1} \underline H_k$. We refer to lines $6, 8, 12, 25, 27$ and $36$ of
Algorithm \ref{alg:pseudobarth}.

It remains to show how to compute $\underline \tau_{k-m+1}$ (after having computed $\mb v_{k+1},\rho_{k-m+1},W_{k-m}$) and relate the term
$\widehat W_{k-m-1} \underline \mu_k$ in \eqref{eq:BM_recurrence} to the term
$V_{k-m} G_{k-m} F^* \mb v_k = V_{k-m} V_{k-m}^* G \underline \mu_k$ of our short recurrence of Proposition~\ref{prop:short_rec}.
In fact, at this place the authors of \cite{barthmanteuffel} require an additional delay in the recurrence \eqref{eq:BM_recurrence} by replacing $m$ by $m':=m+m_3\geq m$, which
is possible according to \eqref{eq:BM_generator}.
According to \cite[Eqn.\ (5.2)]{barthmanteuffel} and \eqref{eq:BM_generator}, the computation of $\underline \tau_{k-m'+1}$ is done by solving the
system
\begin{equation} \label{eq:BM_tau}
        \underline \tau_{k-m'+1}^* [ \underline \mu_{k-\ell} ]_{\ell=0,...,m_{3}-1} =
        [ H_{k-m'+1,k-\ell} - \underline \rho_{k-m'+1}^* \underline \eta_{k-\ell} ]_{\ell=0,...,m_{3}-1}.
\end{equation}
We refer to line $16$ of Algorithm \ref{alg:pseudobarth}.
However, there is a possible problem with this system which is not discussed in \cite{barthmanteuffel}. As noticed after \cite[Eqn.\ (5.2)]{barthmanteuffel}, it is consistent, but one may
not insure that the matrix is invertible, i.e., we might have several solutions, each of them being a generator suitable for $H_k$, but not necessarily for $H_N$.
This is a general problem with computing generators for $H_N$ in a recursive manner: there is no guarantee that $m_3=\mbox{rank}(F^*)$ is
equal to $\mbox{rank}(F^*V_N)$, and thus whether the minimal number of generators for $H_N$ is equal to $m_2+m_3$. In addition, at stage $k$ of a
recursive computation, it might happen that the minimal number of generators for $H_k$ is strictly lower, i.e., $\mbox{rank}(F^*V_k)<\mbox{rank}(F^*V_N)$,
i.e., we should have a $m_3$ depending on $k$. Finally, the matrix of coefficients is just obtained by picking the last $m_3$ columns of $F^*V_k$ which might also lower the rank.
However, going through the proof of \eqref{eq:BM_generator} we can derive an explicit formula for $\underline \tau_j$. From \cite[Eqn.\ (4.14)]{barthmanteuffel} it can be deduced that
\begin{equation}
 \underline \tau_j = G^\ast q(A) \alpha_j(A) \mb v_1. \label{eq:BM:tau}
\end{equation}
Combining \eqref{eq:BM:tau} and \eqref{eq:BM_rho}, we obtain
\begin{eqnarray}
 \underline \tau_j &=& G^*\left(\mb v_j - (\mb v_1, \ldots, \mb v_{m_2-1}) \underline \rho_j \right) \nonumber \\
                   &=& G^* \mb v_j - G^* V_{m_2} \underline \rho_j. \label{eq:BM_new1}
\end{eqnarray}
Note that \eqref{eq:BM_new1} could be used to compute $\underline \tau_{k-m+1}$ without introducing the
additional delay in \eqref{eq:BM_recurrence}.
Inserting \eqref{eq:BM_new1} into \eqref{eq:BM_qvectors} gives the explicit formula
\begin{align*} \label{eq:BM_new2}
   \widehat W_{k-m-1} &= \sum_{j=1}^{k-m} \mb v_j \left( \mb v_j^\ast G - \underline \rho_j^\ast V_{m_2}^\ast G \right) \\
                      &= V_{k-m} V_{k-m}^* G - W_{k-m-1} V_{m_2}^* G .
\end{align*}

In the following remark we intend to compare our approach with that of Barth \& Manteuffel  \cite{barthmanteuffel}.

\begin{remark}\begin{enumerate}

\item[(a)] Both approaches heavily use the fact that a certain upper right part of the Hessenberg matrix is of rank at most $m_2+m_3$. In other words, one is able to express
$A \mb v_k$ as a linear combination of $m_2+m_3$ correction vectors plus a linear combination of the last $m$ or $m'=m+m_3$ columns of $V_{k+1}$, a kind of corrected ``short'' recurrence.
Notice however that our recurrence is ``shorter'' if $m_3>0$.
\item[(b)] In our approach, $m_3$ correction vectors are explicitly given (the term $\mb v_k'-A \mb v_k$) and can be updated explicitly. They are not necessarily linearly independent.
The other $m_2$ correction vectors are identified as GMRES residuals for shifted systems, allowing for easy updating.

    In contrast, Barth \& Manteuffel compute explicitly the four sequences of generators of the low-rank structure of $H_N$, given in \eqref{eq:BM_generator}.
    Notice, however, that at the $k$th step of the algorithm one can only deduce generators for $H_k$ and not for $H_N$. As mentioned above, in order not to be
    obliged to correct generators found earlier, there should be an additional assumption on $\mbox{rank}(F^*V_k)$ not mentioned by the authors.
    However, there is a variant of the Barth \& Manteuffel approach: instead of using
    \eqref{eq:BM_tau} requiring a delay in the recurrence relation \eqref{eq:BM_recurrence}, one can use \eqref{eq:BM_new1}, which is not mentioned in \cite{barthmanteuffel}, to compute
    $\underline \tau_{k-m+1}$ just after having computed $\underline \rho_{k-m+1}$.
\item[(c)] In our approach, for finding the coefficients of the GMRES correction vectors, we suggest to solve a least square problem, the matrix of coefficients $M_k$ having as
columns these $m_2$ (normalized) GMRES correction vectors. Notice that $M_k$ has full column rank (since $V_{m_{2}}^* M_k$ has), but might be ill-conditioned. Thus standard techniques
(SVD dropping small singular values, or $QR$
decomposition with column pivoting and threshold) can be applied, where the residual error in solving this least-square problem leads to a loss of orthogonality for $v_{k+1}$ of the same order.

    In contrast, Barth \& Manteuffel suggest one of the missing generators by solving system
    \eqref{eq:BM_tau}.
    The computation of the other generator $\underline \rho_{j+1}$ is quite involved and requires the knowledge of the whole $j$th column of
    the Hessenberg matrix $\underline H_k$ (which is not necessarily computed using our approach).

\item[(d)] If GMRES is converging fast, we believe that the normalization \eqref{eq:link_w} is not appropriate since
    $$
       \|  V_{m_2}^* M_k \| \ll
       \| M_k \| \approx 1.
    $$
    Also the $\underline \eta_k$ are very small due to the above-mentioned decay property of the entries of our Hessenberg matrix.
    %However, since in the original paper there is a different normalization, these scaling issues might be less pronounced in the original version of the algorithm.
\end{enumerate}
\end{remark}

\begin{algorithm}[ht]
\scriptsize
 \KwData{$A \in \mathbb{C}^{n\times n}$, $\underline{ \mb p}_0$ with $A^\ast q_m(A) - p_\ell(A) = Q_B(A)$, $p_\ell$ of degree $\ell$, $p_m$ of degree $m$ and $Q_B(A)$ of rank $\kappa$,
 $\{\underline \phi_0, \ldots, \underline \phi_{\kappa-1}\}$ basis of the column space of $Q_B(A)$,
 $\theta = \left\{ \begin{array}{cc} \kappa - 1, & \text{if}\,\,\, m > \ell \\
                                     \kappa - \ell + m, & \text{if}\,\,\, m \leq \ell
                   \end{array}\right.$, $\phi = \left\{ \begin{array}{cc} j-\theta+1, & \text{if}\,\,\, m > \ell \\
                                                                          \text{min}\{j-(\ell-m),j-\theta+1 \}, & \text{if}\,\,\, m \leq \ell
                                                        \end{array}\right.$
$\underline{ c \mb p}_0 = [\begin{array}{cccc} 1 & 0 & \ldots & 0 \end{array} ]_{m+1}^T$, $\underline \rho_0 = [\begin{array}{cccc}1 & 0 & \ldots & 0 \end{array}]_m^T$ }
 \KwResult{Orthogonal Krylov basis $\underline{\mb p}_0, \underline{\mb p}_1, \ldots$}

 \For{$j=0, \ldots, \theta-1$}{
 $\mb{\underline p}_{j+1} = A \mb{\underline p}_j - \sum_{i=0}^j \sigma_{i,j} \mb{\underline p}_j, \,\,\,
 \sigma_{i,j} = \frac{\mb{\underline p}_i^\ast A \mb{\underline p}_j }{\mb{\underline p}_i^\ast \mb{\underline p}_i}$ \;
     \eIf{$j<m$}{
   $\underline{c\mb p}_{j+1} = \underline{cA \mb p}_j - \sum_{i=0}^j \sigma_{i,j} \underline{c p}_j, \,\, \underline{c A \mb p}_j =
   [\begin{array}{cc}0 & \underline{c \mb p}_j \end{array} ]_{m+1}^T$\;
    \eIf{$j<m-1$}{$\underline \rho_{j+1} = [\begin{array}{ccccccc}0& \ldots & 0 & 1 & 0 & \ldots & 0\end{array}]_m^T$\;}{$\text{solve}\,\,\,
    \left[\begin{array}{cccc} | & | & & | \\
     \underline{c \mb p}_0 & \underline{c \mb p}_1 & \ldots & \underline{c \mb p}_m \\
     | & | & & | \end{array} \right][\begin{array}{ccc}\gamma_0 & \ldots & \gamma_m \end{array}]^T = [\begin{array}{ccc}b_0 & \ldots & b_m \end{array}]^T \,\,\, \text{for}
     \,\,\, \underline \gamma, \,\,\,
    \underline \rho_{j+1} = \left[\begin{array}{ccc}-\frac{\gamma_0}{\gamma_m} & \ldots & - \frac{\gamma_{m-1}}{\gamma_m} \end{array} \right]$\;}
   }{$\underline c_j = H_{m+1,m} \underline \rho_j, \,\,\, \underline{rA \mb p}_j = c_0 \underline \rho_0 + \ldots + c_m \underline \rho_m$\;
   $\underline \rho_{j+1} = \underline{rA \mb p}_j - \sum_{i=0}^j \sigma_{i,j} \underline \rho_i$\;}
 }
 \For{$j = \theta, \ldots,$}{
 $\text{solve}\,\,\, [\begin{array}{ccc}\underline \mu_j & \ldots & \underline \mu_{j-\kappa+1} \end{array}]^T  \underline{\bar \tau}_{j-\theta} =
 \left[ \begin{array}{c}  \underline{\mb p}_{j-\theta}^\ast  A \underline{\mb p}_j- \underline \rho_{j-\theta}^\ast \underline \eta_j  \\
                         \vdots \\
                         \underline{\mb p}_{j-\theta}^\ast A \underline{\mb p}_{j-\kappa+1} - \underline \rho_{j-\theta}^\ast \underline \eta_{j-\kappa+1} \end{array} \right]
 \,\,\,\text{for} \,\,\, \underline{\bar \tau}_{j-\theta}$\;
 $\underline{ \mb q}_{j-\theta_i} = \frac{(\underline{\bar \rho}_{j-\theta})_i}{\underline{\mb p}_{j-\theta}^\ast \underline{\mb p}_{j-\theta} }
 \underline{ \mb p}_{j-\theta}+ \underline{\mb q}_{{j-\theta-1}_i}, \,\,\, i = 0, \ldots, m-1$\;
 $\underline{\hat{ \mb q}}_{j-\theta_i} = \frac{(\underline{\bar \tau}_{j-\theta})_i}{\underline{\mb p}_{j-\theta}^\ast \underline{\mb p}_{j-\theta} }
 \underline{ \mb p}_{j-\theta}+ \underline{\hat{\mb q}}_{{j-\theta-1}_i}, \,\,\, i = 0, \ldots, \kappa-1$\;
 $\underline{r \mb q}_{j-\theta_i} = \frac{(\underline{\bar \rho}_{j-\theta})_i}{\underline{\mb p}_{j-\theta}^\ast \underline{\mb p}_{j-\theta} }
 \underline{ \rho}_{j-\theta}+ \underline{r \mb q}_{{j-\theta-1}_i}, \,\,\, i = 0, \ldots, m-1$\;
 $\underline{r\hat{ \mb q}}_{j-\theta_i} = \frac{(\underline{\bar \tau}_{j-\theta})_i}{\underline{\mb p}_{j-\theta}^\ast  \underline{\mb p}_{j-\theta}}
 \underline{ \rho}_{j-\theta}+ \underline{r \hat{ \mb q}}_{{j-\theta-1}_i}, \,\,\, i = 0, \ldots, \kappa-1$\;
 \eIf{$j<m$}
 {$\underline{ \mb p}_{j+1} = A \underline{ \mb p}_j - \sum_{i=0}^j \sigma_{i,j} \underline{ \mb p}_j, \,\,\,
 \sigma_{i,j} = \frac{\mb{\underline p_i}^\ast A \mb{\underline p_j}}{\mb{\underline p_i}^\ast \mb{\underline p_i}}$\;
 $\underline{c\mb p}_{j+1} = \underline{cA \mb p}_j - \sum_{i=0}^j \sigma_{i,j} \underline{c p}_j, \,\, \underline{c A \mb p}_j =
   [\begin{array}{cc}0 & \underline{c \mb p}_j \end{array} ]_{m+1}^T$\;
  \eIf{$j<m-1$}{$\underline \rho_{j+1} = [\begin{array}{ccccccc}0& \ldots & 0 & 1 & 0 & \ldots & 0\end{array}]_m^T$\;}{$\text{solve}\,\,\,
    \left[\begin{array}{cccc} | & | & & | \\
     \underline{c \mb p}_0 & \underline{c \mb p}_1 & \ldots & \underline{c \mb p}_m \\
     | & | & & | \end{array} \right][\begin{array}{ccc}\gamma_0 & \ldots & \gamma_m \end{array}]^T = [\begin{array}{ccc}b_0 & \ldots & b_m \end{array}]^T \,\,\, \text{for}
     \,\,\, \underline \gamma,\,\,\,
    \underline \rho_{j+1} = \left[\begin{array}{ccc}-\frac{\gamma_0}{\gamma_m} & \ldots & - \frac{\gamma_{m-1}}{\gamma_m} \end{array} \right]$\;}
 }
 {$\underline \eta_j = [\begin{array}{ccc} \underline{\mb p}_0^\ast A \underline{\mb p}_j  & \ldots &
 \underline{\mb p}_{m-1}^\ast A \underline{\mb p}_j  \end{array} ]_m^T$\;
 $\underline \mu_j = [\begin{array}{ccc} \underline \phi_0^\ast \underline{\mb p}_j & \ldots &
 \underline \phi_{\kappa-1}^\ast \underline{\mb p}_j  \end{array} ]_\kappa^T$\;
 $\underline{ \mb y}_{j+1} = A \underline{ \mb p}_j - [\begin{array}{ccc} \underline{\mb q}_{j-\theta_0} & \ldots & \underline{\mb q}_{j-\theta_{m-1}} \end{array} ]
 \underline \eta_j - [\begin{array}{ccc} \underline{\mb {\hat q}}_{j-\theta_0} & \ldots & \underline{\mb {\hat q}}_{j-\theta_{m-1}} \end{array} ]
 \underline \mu_j$\;
 $\underline{ \mb p}_{j+1} = \underline{ \mb y}_{j+1} - \sum_{i=\phi}^j \hat t_{i,j} \underline{ \mb p}_i, \,\,\,
 \hat t_{i,j} = \frac{\underline{ \mb p}_i^\ast \underline{\mb y}_{j+1} }{\underline{\mb p}_i^\ast \underline{\mb p}_i }$\;
 $\underline c_j = H_{m+1,m}\underline \rho_j, \,\,\, \underline{r A \mb p}_j = c_0 \underline \rho_0 + \ldots + c_m \underline \rho_m$\;
 $\underline{r \mb y}_{j+1} = \underline{rA \mb p}_j - \sum_{i=0}^{m-1} (\underline \eta_j)_i \underline{r \mb q}_{{j-\theta}_i} - \sum_{i=0}^{\kappa-1} (\underline \mu_j)_i
 \underline{r \hat {\mb q}}_{{j-\theta}_i}$\;
 $\underline \rho_{j+1} = \underline{r \mb y}_{j+1} - \sum_{i=\phi}^j \hat t_{i,j} \underline \rho_i$\;}

 }

 \caption{Barth-Manteuffel algorithm}
 \label{alg:pseudobarth}
\end{algorithm}

\section{Some special matrices} \label{sec:unitary}
In this section we give special attention to some classes of matrices where we can slightly reduce the computational complexity of Algorithm 1. The first class consists of matrices $A$ which
satisfy an equation of the form
\begin{equation}
 A^\ast - \alpha A - \beta I = FG^\ast, \label{eq:nearlyherm1}
\end{equation}
for some $\alpha, \beta \in \mathbb{C}$. This includes the class of normal matrices of which all but $m_3$ eigenvalues are collinear.
We will address this kind of matrices as \textit{nearly Hermitian matrices}. If $\alpha = 1, \beta = 0$ and $A$ is real, this corresponds to the class of
\textit{nearly symmetric matrices} as discussed in \cite{BeckReich}. However, one can easily check that all results derived in \cite{BeckReich} are also valid for a matrix of the
form \eqref{eq:nearlyherm1}.

The second class consists of matrices $A$ which satisfy an equation of the form
\begin{equation} A^\ast - \alpha I - \beta (A-\delta I)^{-1} = FG^\ast, \label{eq:nearlyuni}
\end{equation}
for some $\alpha, \beta, \delta \in \mathbb{C}$.
This includes the class of normal matrices of which all but $m_3$ eigenvalues are concyclic. If $\delta = 0$, we speak of \textit{nearly unitary matrices}, if $\delta \not= 0$ we speak of
\textit{nearly shifted unitary matrices}.

The class of matrices satisfying \eqref{eq:nearlyherm} or \eqref{eq:nearlyuni} include all examples of $BML$-matrices known to us which are of practical interest.
By this we mean,
matrices that are suitably large with respect to the quantities $m_1, m_2$ and $m_3$.
More information on matrices satisfying equation \eqref{eq:introeq} can be found in \cite{Liesen}.

\subsection{Nearly Hermitian matrices}
\label{subsec:nh}

Assume the matrix $A$ is nearly Hermitian. Then \eqref{eq:8} reduces to
\begin{equation}
 \mb v_k' = h_{k+1,k} \mb v_{k+1} + h_{k,k} \mb v_k + h_{k-1,k} \mb v_{k-1}. \label{eq:herm}
\end{equation}
Define the vectors $\mb{p}_k^\ast \in \mathbb{C}^{m_3}$ recursively as
\begin{equation}
 \mb{p}_k^\ast := -s_{k-1}(0)\mb{p}_{k-1}^\ast + c_{k-1}(0)\mb{e}_k^\ast G_k, \label{eq:p0}
\end{equation} for $k \geq 2$ and $\mb{p}_1^\ast = G_1$.
By recurrence on $k$ it follows that $\mb{p}_k^\ast = \mb{e}_k^\ast Q_k(0)^\ast G_k$. Therefore, in combination with \eqref{eq:20},
\begin{equation}
 \mb{p}_k^\ast = (-1)^{k-1}\mb{w}_{k-1}(0)^\ast V_k G_k. \label{eq:taubis}
\end{equation}
Then \eqref{eq:tau} yields
\begin{eqnarray}
 \tau_k(0) &=& (-1)^{k-1} \mb w_{k-1}(0)^\ast A \mb v_k \nonumber \\
           &=& (-1)^{k-1} \mb{w}_{k-1}(0)^\ast (\mb{v}_k' + V_k G_k F_k^\ast \mb{e}_k) \nonumber \\
           &=& (-1)^{k-1} \mb w_{k-1}(0)^\ast \mb v_k' + (-1)^{k-1} \mb w_{k-1}(0)^\ast V_k G_k F_k^\ast \mb e_k \nonumber \\
           &=& h_{k,k} c_{k-1}(0) - h_{k-1,k}c_{k-2}(0) s_{k-1}(0) + \mb p_k^\ast F_k^\ast \mb e_k, \label{eq:tauherm}
\end{eqnarray}
the latter equality because of \eqref{eq:herm} and \eqref{eq:taubis}. Expression \eqref{eq:tauherm} can now be used to compute $\tau_k(0)$ instead
of \eqref{eq:tau}, reducing the computational complexity\footnote{Expression
  \eqref{eq:tauherm} was also proved alternatively By Beckermann and Reichel \cite[Proposition 4.2]{BeckReich}.}.
%, without making notice of
%\eqref{eq:tau}. Hence,
%The derivation above offers, however, an alternative proof of \cite[Proposition 4.2]{BeckReich}.}.

\subsection{Nearly unitary matrices}
\label{subsec:num}
Assume the matrix $A$ is nearly unitary.
Then \eqref{eq:8} reduces to
\begin{equation}
 \mb{v}_k' = a_{1,k} \mb{w}_{k-1}(0) + h_{k+1,k}\mb{v}_{k+1}, \label{eq:univec}
\end{equation}
where $a_{1,k} = \mb{w}_{k-1}(0)^\ast \mb{v}_k'$ and $h_{k+1,k}$ such that $\mb{v}_{k+1}$ is of unit length.
Again we make use of the vector $\mb p_k^\ast$ as defined in \eqref{eq:p0}.
Then because of \eqref{eq:taubis}, \eqref{eq:tau} yields
\begin{eqnarray}
 \tau_k(0) &=& (-1)^{k-1} \mb{w}_{k-1}(0)^\ast A \mb{v}_k \nonumber \\
           &=& (-1)^{k-1} \mb{w}_{k-1}(0)^\ast (\mb{v}_k' + V_k G_k F_k^\ast \mb{e}_k) \nonumber \\
           &=& (-1)^{k-1} a_{1,k} + \mb{p}_k^\ast F_k^\ast \mb{e}_k. \label{eq:tauuni}
\end{eqnarray}
Expression \eqref{eq:tauuni} can now be used to compute $\tau_k(0)$ instead of \eqref{eq:tau}.

\subsection{Nearly shifted unitary matrices}
\label{subsec:nsum}
Assume the matrix $A$ is nearly shifted unitary. Then \eqref{eq:8} reduces to
\begin{equation}
 \mb v_k' = a_{1,k} \mb w_{k-2}(\delta) + h_{k,k} \mb v_k + h_{k+1,k} \mb v_{k+1}, \label{eq:shiftedrec}
\end{equation}
where $a_{1,k} = \mb w_{k-2}(\delta)^\ast \mb v_k'$ and $h_{k,k}$, $h_{k+1,k}$ are entries of the corresponding Hessenberg matrix.
From \eqref{eq:update} we deduce that
\begin{equation}
 \mb v_{k-1}^\ast   \mb w_{k-2}(\delta)= (-1)^{k-2} \overline{c_{k-2}(\delta)}. \label{eq:nearly1}
\end{equation}
Hence, due to \eqref{eq:update} and \eqref{eq:shiftedrec},
\begin{eqnarray}
 \mb w_{k-2}(\delta)^\ast \mb v_{k-1}' &=& s_{k-2}(\delta) \mb w_{k-3}(\delta)^\ast \mb v_{k-1}' + (-1)^{k-2} c_{k-2}(\delta) \mb v_{k-1}^\ast \mb v_k' \nonumber \\
                                       &=& s_{k-2}(\delta) a_{1,k-1} + (-1)^{k-2} c_{k-2}(\delta) a_{1,k} \mb v_{k-1}^\ast \mb w_{k-2}(\delta) \nonumber \\
                                       &=& s_{k-2}(\delta) a_{1,k-1} + a_{1,k} |c_{k-2}(\delta)|^2. \label{eq:nearly2}
\end{eqnarray}
Finally, we know that
\begin{equation}
 (-1)^{k-1}\mb{w}_{k}(\delta)^\ast V_k G_k = s_k(\delta) \mb p_k^\ast(\delta), \label{eq:nearly3}
\end{equation}
with $\mb p_k^\ast(\delta) \in \mathbb{C}^{m_3}$ recursively defined as
\begin{equation}
 \mb{p}_k^\ast := -s_{k-1}(\delta)\mb{p}_{k-1}^\ast + c_{k-1}(\delta)\mb{e}_k^\ast G_k, \label{eq:p}
\end{equation} for $k \geq 2$ and $\mb{p}_1^\ast = G_1$.
Making use of \eqref{eq:nearly1},  \eqref{eq:nearly2} and \eqref{eq:nearly3}, \eqref{eq:tau} yields
\begin{eqnarray*}
 \tau_{k-1}(\delta) &=& (-1)^{k-2} \mb w_{k-2}(\delta)^\ast (A - \delta I_n) \mb v_{k-1} \\
                    &=& (-1)^{k-2} \mb w_{k-2}(\delta)^\ast \left(\mb v_{k-1}' + V_{k-2}G_{k-2}F_{k-1}^\ast \mb e_{k-1} -\delta  \mb v_{k-1} \right) \\
                    &=& (-1)^{k-2}s_{k-2}(\delta) a_{1,k-1} + (-1)^{k-2} a_{1,k} |c_{k-2}(\delta)|^2 \\ &\phantom{=}& - s_{k-2}(\delta) \mb{p}_{k-2}(\delta)^\ast F_{k-1}^\ast \mb{e}_{k-1} - \delta c_{k-2}(\delta).
\end{eqnarray*}
As before, the above expression can now be used to compute $\tau_{k-1}(\delta)$ instead of \eqref{eq:tau}.

\section{Numerical examples} \label{sec:num}

In this section we will compare our fast Arnoldi algorithm with the one of Barth-Manteuel and classical
Arnoldi. We focus especially on the orthogonality of the obtained Arnoldi vectors. The orthogonality in
the forthcoming figures is measured by a method described originally by Paige [6, 15]. Given
$V_k^* V_k - I = U_k+U_k^*$ with $U_k$ strictly upper triangular, we define $S_k = (I + U_k)^{-1} U_k$. The
norm of $S_k$ is used as an orthogonality measure for the columns of $V_k$, i.e., $\| S_k \|\in [0, 1]$ where
$\| S_k \|=0$ when they are orthonormal and $\| S_k \|=1$ when they are linearly dependent [15]. To
make the fairest possible comparison with the Barth-Manteuffel algorithm, we implemented
their pseudocode as stated in [2] and recalled in Algorithm 2. However, their pseudocode returns an orthogonal basis,
while recurrence relation (3.6) returns an orthonormal basis. Therefore, we have normalized
these vectors first.

We will start in \S 6.1 and \S 6.2 by discussing the special case of nearly unitary $A$ (with $\delta=0$) and nearly shifted unitary $A$ (with shift $\delta\neq 0$), where we replaced in our BML-Arnoldi algorithm formula (2.23) for the computation of $\tau_k(\delta)$ by the less expensive formulas described in \S 5.2, and \S 5.3, respectively. Subsequently we report in \S 6.3 about an example of a nearly Hermitian matrix discussed already in [6].

Quite often there is some correlation between loss of orthogonality between Arnoldi  vectors and convergence of the GMRES residual $r_k(\delta)$ of the shifted system $(A-\delta I)x=b$ with starting vector $x_0=0$. This phenomenon is probably related to the decay properties mentioned in Remark 2.7. We therefore draw in each of the figures below the relative GMRES residual
\begin{equation}\label{xxx}
         \frac{\| r_k(\delta) \|}{\| r_0(\delta) \|} = \frac{1}{\sigma_k(\delta)} = \prod_{j=1}^k s_j(\delta) ,
\end{equation}
the last identity following from (2.8). Notice that the quantities $s_j(\delta)$ are
already computed in the BML-Arnoldi algorithm in the case $m_3=1$ of \S 6.1 and \S 6.2,
whereas in \S 6.3 we have to add the computation of $s_j(\delta)$, here  for $\delta=0$,
following the formulas given in \S 5.1. One may understand \eqref{xxx}  as the recursive
computation of the GMRES residuals following some progressive residual scheme, where the
underlying least squares problem is solved by successive Givens rotations.
We will refer to this residual in the forthcoming figures as the \emph{progressive residual}. However, due to
loss of orthogonality, it might be that these progressive residuals are badly
computed. This is why each time we display also the "exact" relative GMRES residual,
obtained by computing the $k$th iterate of GMRES for the shifted system $(A-\delta I)x=b$
with starting vector $x_0=0$ via the black box routine of Matlab (which does not use our
Arnoldi vectors but recomputes them via full Arnoldi, and solves the least squares problem
via Householder transforms).
It turns out that, in all our numerical experiments, that when both Arnoldi and fast
Arnoldi behave well, that the progressive residual and the GMRES residual
exhibit the same convergence history. %align nicely.

All computations were carried out in Matlab R2015a. As a starting vector for the Krylov
subspace we always consider a vector $b$ that has normally distributed random entries with
mean zero and variance one.

\subsection{Perturbed diagonal  and unitary matrices}
We consider $200 \times 200$ diagonal matrices for which all but $m_3$ eigenvalues lie on a circle. Clearly, such matrices satisfy equation \eqref{eq:introeq} with $m_1 =
m_2 = 1$.
%Figure \ref{fig:shifted_uni} depicts the results for three such matrices.
We
considered various cases; for each case we show
%left plots shows
the eigenvalues of the matrix
and a comparison of the orthogonality of the computed Arnoldi vectors for classical
Arnoldi, Barth-Manteuffel, and the fast Arnoldi method.
% (based on Section~\ref{subsec:num},
%for $\delta=0$
%and Section~\ref{subsec:nsum}, for $\delta\neq 0$). Also the progressive
%GMRES residual, i.e., the values of $\mb{w}_k(\delta)$ and the actual GMRES residuals,
%obtained after running GMRES, on the dense matrix are plotted.
%various
%algorithms together with the progressive and gmres residuals.
The legend is plotted in
Figure~\ref{fig:part_circle} and is identical for all similar graphs in this section.
%Overall, we will see that the progressive residual is robust w.r.t.\ the other approaches.

% the orthogonality of the Arnoldi vectors when computed with standard (modified) Arnoldi, recurrence relation \eqref{eq:8} and the Barth-Manteuffel algorithm.

\begin{enumerate}
\item In the first experiment, see Figure~\ref{fig:part_circle}, we consider eigenvalues
  on three quarter of the unit circle.  Clearly the full Arnoldi and fast Arnoldi perform
  best and in all the tests we ran the orthogonality of the computed vectors was
  comparable. Other experiments revealed that Barth-Manteuffel performed just slightly
  worse when considering eigenvalues distributed over the entire unit circle, the
  performance of Barth-Manteuffel started to degrade when segments were excluded from the
  unit circle.  The progressive residual
%($\mb{w}_k(0)$ from \ref{subsec:num})
seems to align almost perfectly with the GMRES residual. We have also tested various radii
and similar conclusions hold when the radius of the circle is changed.

\item In the second experiment, see Figure~\ref{fig:shifted_circle}, we have shifted
  the unit circle in the complex plane. Barth-Manteuffel  seems
  to have problems with this case, Arnoldi, and the fast Arnoldi method
%($\mb{w}_k(0)$ from \ref{subsec:num})
on the other hand
  exhibit good accuracy. The progressive residual and the GMRES residual align again
  almost perfectly.
%The GMRES residual plotted was the one of the shifted system and
%  decays rapidly, but still with nice orthogonal vectors.

\item In a third experiment, see Figure~\ref{fig:out_circle}, all eigenvalues except for
  two are located on the unit circle. In this case Barth-Manteuffel outperforms our code
  slightly. We note that the location of the eigenvalues outside the circle does not have
  a significant impact on the overall picture of the accuracy.
 Tests revealed also that if one would shift the midpoint or exclude eigenvalues out of parts of the circle the fast Arnoldi method would outperform Barth-Manteuffel.
\end{enumerate}
Overall we can conclude that the Arnoldi method is the most accurate one and the fast Arnoldi
method is also typically quite close. The Barth-Manteuffel algorithm, however, exhibits
quick loss of orthogonality when the circle is shifted or the eigenvalues do not span the
entire circle. Moreover, the progressive residual and the GMRES residual align almost perfectly.

\begin{figure}[h]
 \centering
 \includegraphics[]{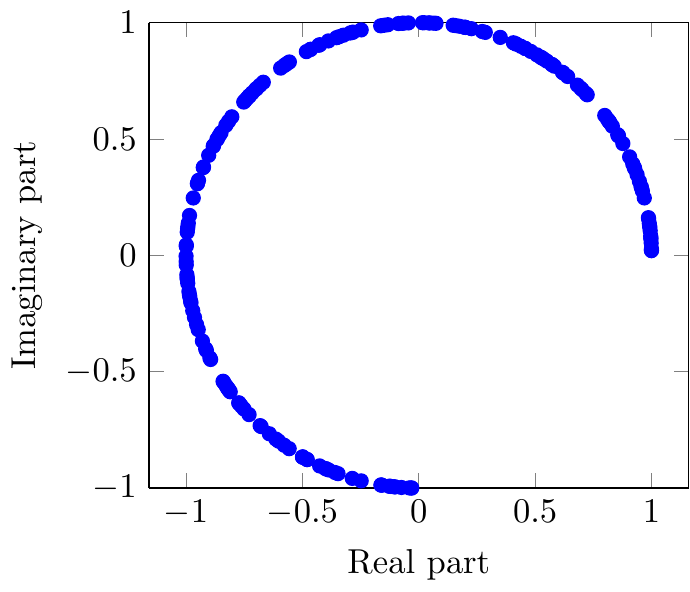}
 \includegraphics[]{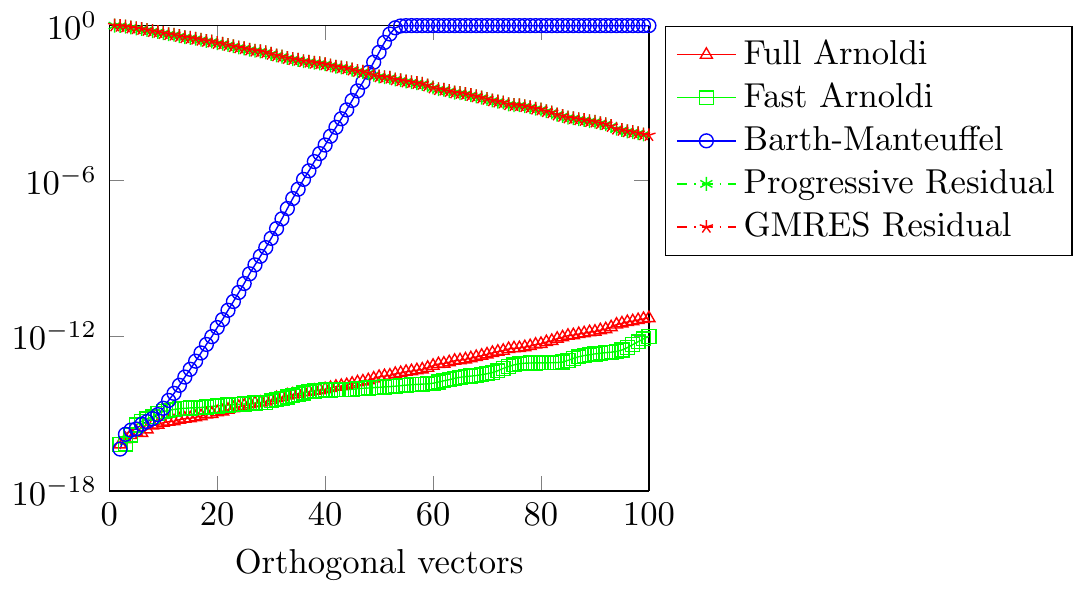}
\caption{The orthogonality, measured by $\|S_k\|$ as in \cite{paige}, of the computed Arnoldi vectors
  for classical Arnoldi, Barth-Manteuffel, and the fast Arnoldi, for
  a matrix having its eigenvalues randomly distributed over three quarter of
  the unit circle. The
 progressive residual as well as the residual of GMRES are also depicted, we observe an
 almost identical behavior for both residuals. }
%
%and residual plot for
%The matrix under consideration}
% Standard (modified) Arnoldi (solid line), the Barth-Manteuffel algorithm (dotted line) and recurrence relation \eqref{eq:8} (dashed line) applied
% to diagonal matrices of size $200$ with almost all eigenvalues on a circle and a randomly
% generated starting vector.
%}
\label{fig:part_circle}
\end{figure}

\begin{figure}[h]
 \centering
 \includegraphics[]{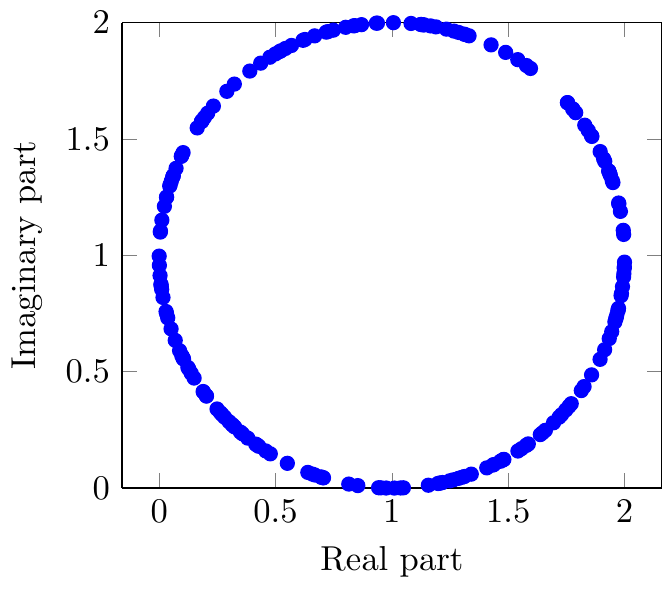}
 \includegraphics[]{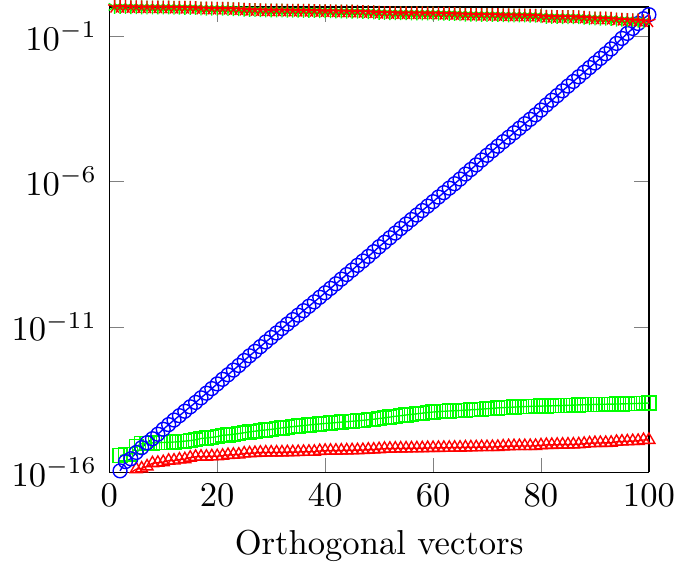}
\caption{The orthogonality, measured as in \cite{paige}, of the computed Arnoldi vectors
  for classical Arnoldi, Barth-Manteuffel, and the fast Arnoldi method, for a
  matrix having its eigenvalues distributed over a shifted unit circle. The
 progressive residual as well as the residual of GMRES are depicted and align nicely. The legend is
 identical to the one used in Figure~\ref{fig:part_circle}.
%Orthogonality and residual plot for eigenvalues distributed over
%a shifted unit circle.
}
% Standard (modified) Arnoldi (solid line), the Barth-Manteuffel algorithm (dotted line) and recurrence relation \eqref{eq:8} (dashed line) applied
% to diagonal matrices of size $200$ with almost all eigenvalues on a circle and a randomly
% generated starting vector.
%}
\label{fig:shifted_circle}
\end{figure}

\begin{figure}[h]
 \centering
 \includegraphics[]{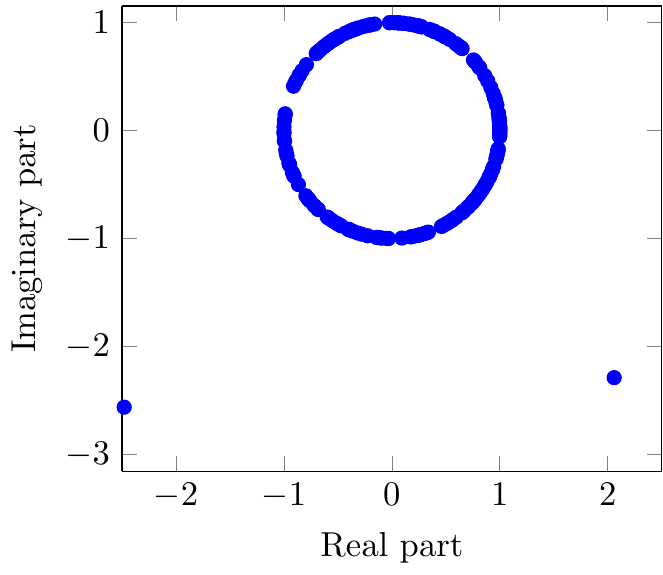}
 \includegraphics[]{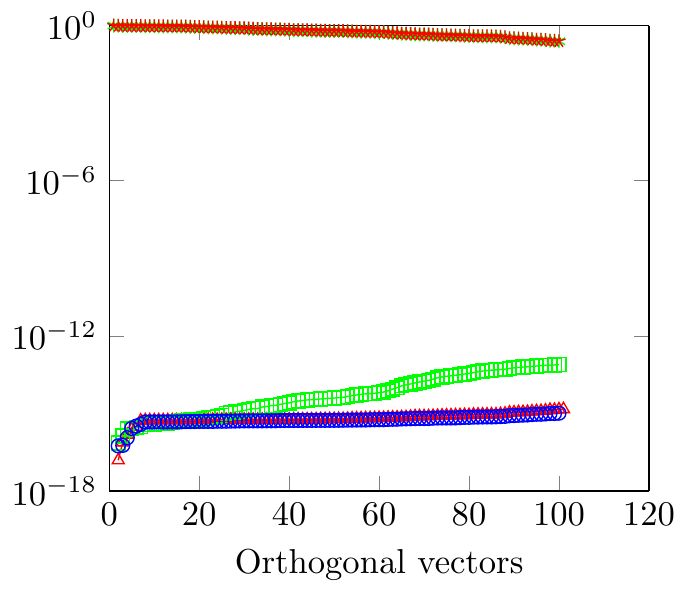}
\caption{The orthogonality, measured as in \cite{paige}, of the computed Arnoldi vectors
  for classical Arnoldi, Barth-Manteuffel, and the fast Arnoldi method, for
  a matrix having all, except two eigenvalues, on the unit circle. The legend is
 identical to the one used in Figure~\ref{fig:part_circle}.
%its eigenvalues distributed over three quarter of
%  the unit circle.
The
 progressive residual as well as the residual of GMRES exhibit the same behavior. }
%
%Orthogonality and residual plot for eigenvalues distributed over
%the unit circle plus two additional eigenvalues.}
% Standard (modified) Arnoldi (solid line), the Barth-Manteuffel algorithm (dotted line) and recurrence relation \eqref{eq:8} (dashed line) applied
% to diagonal matrices of size $200$ with almost all eigenvalues on a circle and a randomly
% generated starting vector.
%}
\label{fig:out_circle}
\end{figure}

We also considered a real non-normal matrix $A$ which is of the form
\begin{equation} A = U + \mb{u} \mb{v}^*, \label{eq:uniplus} \end{equation}
for some vectors $\mb u$ and $\mb v$ and a unitary matrix $U$. Using the Sherman-Morrison formula \cite{b037}, equation \eqref{eq:uniplus} yields
\[A^\ast = A^{-1} + \mb v \mb u^* + \frac{U^\ast \mb u \mb v^* U^\ast}{1+\mb v^* U^\ast \mb u}. \]
Hence,
\[ A^\ast = A^{-1} + FG^\ast, \,\,\, \text{with} \,\,\, F := \left(\begin{array}{cc}
                                                                   \mb v & \frac{1}{1+\mb v^* U^\ast \mb u} U^\ast \mb u
                                                                  \end{array} \right), \,\,\, G := \left(\begin{array}{cc}
                                                                  \mb u & U \mb v
                                                                  \end{array} \right).\]
The orthogonality of the computed Arnoldi vectors was examined
%is depicted in Figure \ref{fig:uniplus}
for a $100 \times 100$ matrix $A$ of the form \eqref{eq:uniplus} where
the unitary matrix $U$ and the vectors $\mb u$ and $\mb v$ are randomly generated.
In this case the orthogonality behaved similar to Figure~\ref{fig:part_circle}, implying
that the fast Arnoldi behaves similar to Arnoldi and Barth-Manteuffel deteriorates.
% and
%ithe new recursion behaves better than Barth-Manteuffel.

% As seen in Figure \ref{fig:uniplus}, recurrence relation \eqref{eq:8} is in favor with
%respect to Arnoldi and the Barth-Manteuffel algorithm.

%\input{Old_Files/uniplus.tex}

\subsection{Unitary matrix from Quantum Chromodynamics}
We consider a shifted unitary matrix which finds its origin in Quantum Chromodynamics (QCD)\cite{qcd}. QCD is the theory which describes the
fundamental interaction between quarks, which are the building blocks of protons and neutrons. This theory makes use of the Neuberger overlap operator $A=\rho I + \gamma\,
\text{sign}(Q)$, where $\rho$ and $\gamma$ are scalars and $Q$ is the Hermitian
Wilson fermian matrix. As a result the Neuberger overlap operator is a shifted unitary matrix.
To construct the matrix $Q$ a parameter $\kappa$
and a hopping matrix are needed. We
have selected these parameters equal to the ones from
%correspondance with
\cite{qcd}, i.e. $\kappa = 0.2809$ and as hopping matrix \texttt{conf5.0-00.14x4-2600.mtx}
from the Matrix Market\footnote{A repository of test data for use in
comparative studies of algorithms for numerical linear algebra, featuring nearly 500 sparse matrices from a variety of applications, as well as matrix generation
tools and services.}.

To compute $\text{sign}(Q)$, we invoke the
software package designed by Arnold, et al.,
\cite{qcd}, which makes use of the Zolotarev algorithm \cite{Zolotarev}. We choose
parameters  $\rho=2$ and $\gamma=1$ for the
Neuberger overlap operator. %In
%Figure~\ref{fig:qcd} the orthogonality of the columns of $\gamma\, \text{sign}(Q)$ is
%depicted when computed with the Zolotarev algorithm.
On the left of Figure~\ref{fig:qcd}
we have drawn the eigenvalues of the Neuberger overlap operator $A$ and observe that the
density of the  eigenvalues is much higher on the right than on the left.

The orthogonality of the computed Arnoldi vectors is depicted in
Figure~\ref{fig:qcd}. Approximately the first ten iterations of the recurrence relation
\eqref{eq:8} fast Arnoldi and the Barth-Manteuffel algorithm show
comparable accuracy. After that, the orthogonality of the vectors computed with the
Barth-Manteuffel algorithm deteriorates fast at almost the same rate as classical
Arnoldi. Even though the progressive residuals and the GMRES residual align, classical
Arnoldi seems to suffer heavily from loss of accuracy. The fast Arnoldi method clearly
outperforms the other approaches.

\begin{figure}[h!tb]
 \centering
 \includegraphics[]{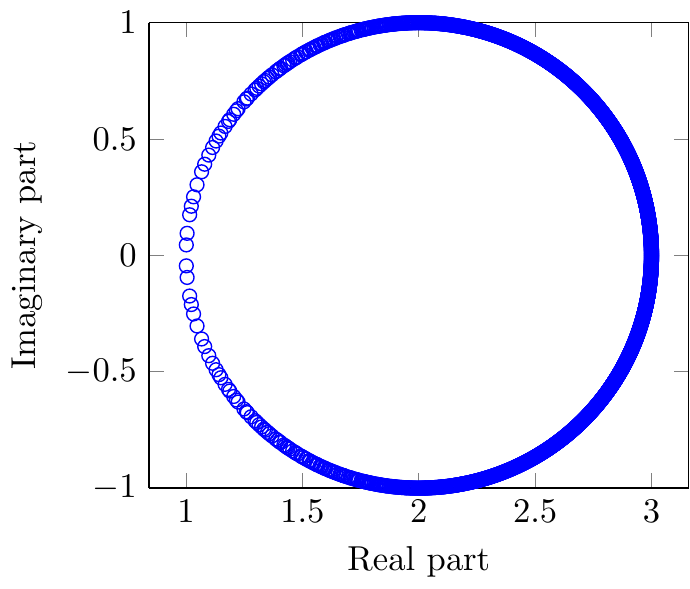}
 \includegraphics[]{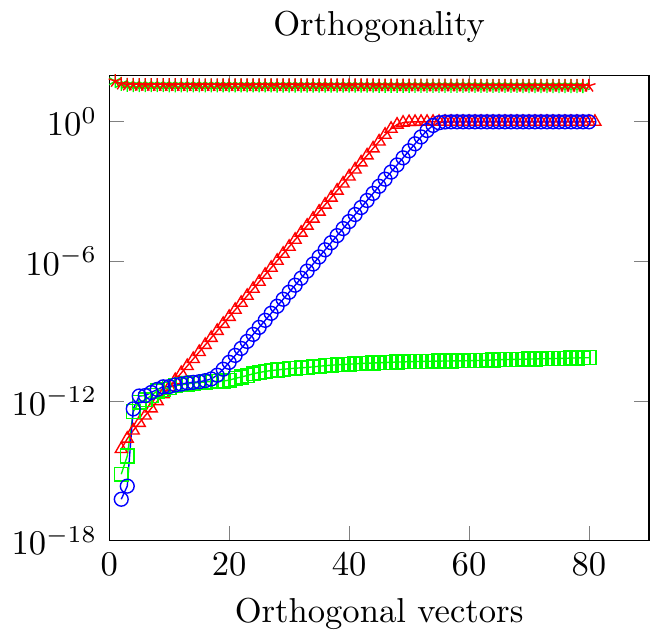}
\caption{The left plot depicts the eigenvalues of the Neuberger operator $A = \rho I +
  \gamma\, \text{sign}(Q)$, with $\rho = 2$ and $\gamma=1$. The right plot shows the accuracy and
  (progressive) residuals
  for classical Arnoldi, Barth-Manteuffel, and the fast Arnoldi algorithms. For the legend we refer to Figure~\ref{fig:part_circle}.}
%TRight plot: standard (modified) Arnoldi (solid line), the Barth-Manteuffel algorithm (dotted line) and recurrence relation \eqref{eq:8} (dashed line) applied to $D$.} \label{fig:qcd}
\label{fig:qcd}
\end{figure}

%\input{qcd.tex}
%\input{qcd_columns}

%\newpage

%\textbf{Example 4}

\subsection{Departure from orthogonality} \,\, Beckermann \& Reichel \cite{BeckReich} proposed a Krylov subspace method for solving a linear system in which the coefficient matrix is nearly
 Hermitian. Their method, based on a short recurrence for generating an orthonormal Krylov basis, is better known as \textit{Progressive GMRES}, shortly
 named PGMRES.
 For nearly Hermitian matrices, this short recurrence coincides with the recurrence relation \eqref{eq:8} described above.

 However, Embree, et al., \cite{embree} showed how in certain cases the PGMRES method exhibits an instability which finds its origin in the loss of orthogonality of the computed
 Arnoldi vectors. A specific class of examples is described and the corresponding
 departure from orthogonality is shown when using the recurrence relation \eqref{eq:8}. In
 the forthcoming experiments we have used the algorithm for nearly Hermitian matrices as
 presented in Section~\ref{subsec:nh}.

 This class of matrices is of the form
 \begin{eqnarray}A = \left(\begin{array}{ccc}
             \Lambda_- &  &  \\
              & \Lambda_+ &  \\
              & & Z
             \end{array}
 \right), \label{eq:classherm} \end{eqnarray}
where \[\Lambda_- = \text{diag}(\lambda_1, \ldots, \lambda_p), \qquad
\Lambda_+ = \text{diag}(\lambda_{p+1}, \ldots, \lambda_{n-2}), \qquad
Z = \left(\begin{array}{cc}
           0 & \gamma \\
           -\gamma & 0
          \end{array}
 \right), \]
with eigenvalues
\begin{enumerate}
 \item $\lambda_1, \ldots, \lambda_p$ uniformly distributed in the interval $[-\beta, -\alpha]$,
 \item $\lambda_{p+1}, \ldots, \lambda_{n-2}$ uniformly distributed in the interval $[\alpha, \beta]$,
 \item $\lambda_{n-1}, \lambda_{n} = \pm \gamma i$.
\end{enumerate}

We take two examples from this class and compare the loss of orthogonality of the recurrence relation \eqref{eq:8}
as predicted in \cite{embree} with
the Barth-Manteuffel algorithm. The orthogonality of the Arnoldi vectors stored in the matrix $V_k$ is depicted in Figure~\ref{fig:herm}.

\begin{figure}[htb]
 \centering
{\begin{tabular}{cc}
\includegraphics[]{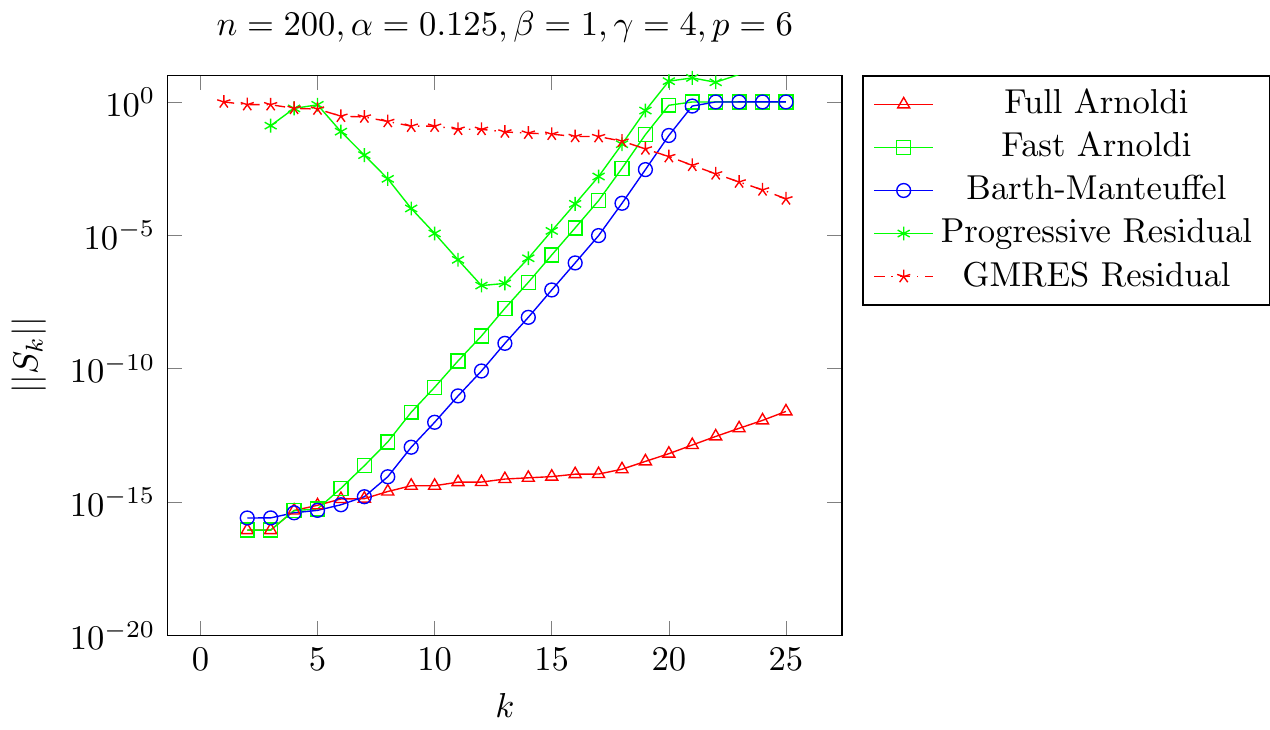}
\end{tabular}}
\caption{
The orthogonality, measured by $\|S_k\|$ as in \cite{paige}, of the computed Arnoldi vectors
  for classical Arnoldi, Barth-Manteuffel, and the fast Arnoldi, for
  a nearly Hermitian matrix of the form \eqref{eq:classherm}.
 The
 progressive residual as well as the residual of GMRES are also depicted, and we see that
 the loss of accuracy in the fast Arnoldi case is related to inaccurate computations in
 the progressive residual vectors.
% we observe an
%  almost identical behavior for both residuals.
% Standard (modified) Arnoldi, the Barth-Manteuffel algorithm, recurrence relation
%   \eqref{eq:8}, and the progressive residual vector applied
% to a nearly Hermitian matrix of the form  and a randomly generated starting vector.
}
\label{fig:herm}
\end{figure}

\begin{figure}
\centering
\includegraphics[]{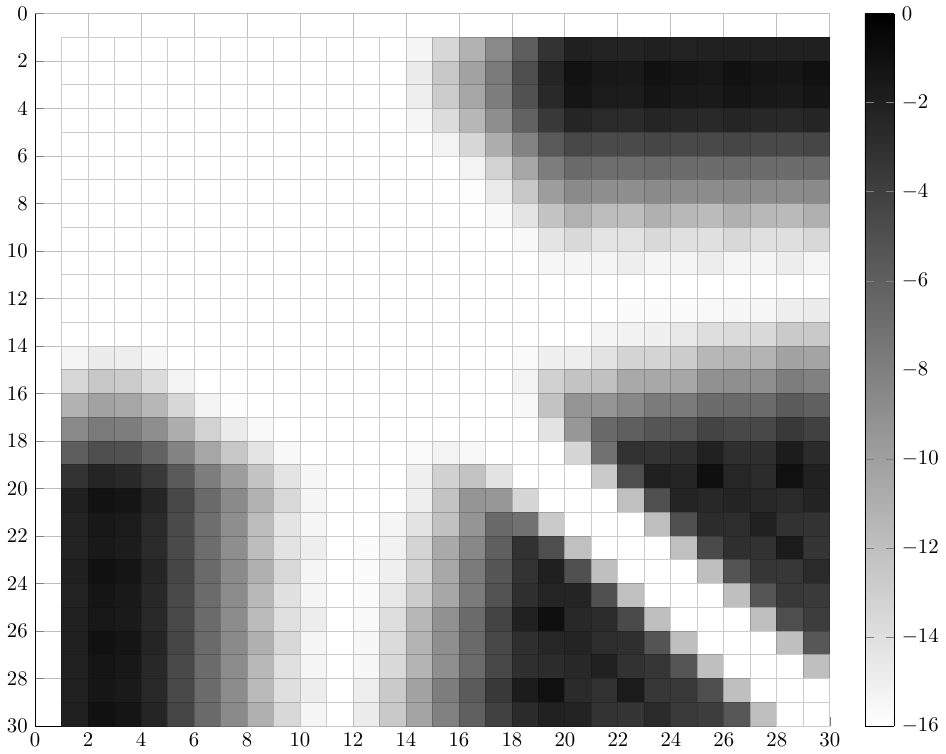}

\caption{The orthogonality between the successively computed vectors via
  progressive GMRES is plotted. The picture shows a log scale of the matrix $V_k^*V_k-I$.}
\label{fig:ortho}

\end{figure}

As seen in Figure~\ref{fig:herm} recurrence relation \eqref{eq:8} gives rise to
significantly less orthogonal vectors than the standard Arnoldi iteration.
%, which is even
%more apparent when the matrix is ill-conditioned (cfr. right plot in
%Figure~\ref{fig:herm}).
However, it may also be observed that the Barth-Manteuffel algorithm suffers from
the same loss of orthogonality as recurrence relation \eqref{eq:8}. We see that the loss
of orthogonality emerges as soon as the progressive residual vectors start to differ
significantly from the actual GMRES residual, explaining the inaccuracies in the computed vectors.
% on
%the left, on the right the progressive residual exhibits a growing norm with a rapid
%increase of loss of orthogonality.
Figure~\ref{fig:ortho} shows the gradual loss of orthogonality
between the vectors. White stands for a perfect orthogonality, black for complete loss,
The colors assigned are, for $10^0$ black (no orthogonality) and for $10^{-16}$ white
(orthogonal up to machine precision).
%The first figure links to the first example of
%Figure~\ref{fig:herm}, the second one to the second example of Figure~\ref{fig:herm}.
In Figure~\ref{fig:ortho} we observe that $v_j$ is numerically orthogonal to $v_k$ for $j,k<=14$, $j
\neq k$, and also at later stages for $|j-k|\leq 2$, as expected from the local
reorthogonalization of our algorithm. However, globally, the orthogonality gets quite
quickly lost, as observed already by Embree et al. \cite{embree}, who suggested Schur complement techniques to
tackle this problem.

We can conclude that in this case both Barth-Manteuffel and fast Arnoldi exhibit a fast
and almost
identical loss of orthogonality.

\section{Conclusion}

An economic variant for the Arnoldi algorithm has been established for matrices whose adjoint is a low-rank perturbation of a rational function of the matrix. In the process, some aspects
of the Arnoldi process are described in terms of orthogonal polynomials. This includes an explicit formula for the unitary factor in the $QR$-decomposition of a Hessenberg matrix and
a decay property of the entries of this Hessenberg matrix which is related to the convergence of the GMRES algorithm.
Also, the existence of a progressive GMRES residual formula has been shown, extending the findings of \cite{BeckReich}. Furthermore, comparisons are made with the
algorithm described by Barth and Manteuffel \cite{barthmanteuffel} for matrices whose adjoint is a low-rank perturbation of a rational function of the matrix, both theoretically
and numerically.

\section{Acknowledgements}

Part of this research has been established during a visit at the university of Science and Technology of Lille in 2014.
We thank Bernd Beckermann and Ana Matos for their hospitality.
%
% \section{When is $A^\ast = p(A)q(A)^{-1} +$ low rank?}
%
% We give a short summary of \cite{Liesen}. It can be shown that a function of the form $r(z)-\overline{z}$, where $r(z)$ is a rational
% function of McMillan degree $s \geq 2$, which is the maximum of the degrees of numerator and denominator, has at moest $5s - 5$ zeros.
% Also, for any $s \geq 2$, there exists a rational function $r$ of McMillan degree $s$ such that $r(z) - \overline{z}$ has $5s-5$ zeros.
% Hence, if we assume that  $A^\ast = p(A)q(A)^{-1}$, we can conclude that if the eigenvalues of $A$ are
% neither collinear or concyclic, then$A$ has only a very small number of distinct eigenvalues, which implies that breakdown will
% occur very soon when running the Arnoldi algorithm.
%
% {\color{red}{Question: How to generate matrices $A$ where $A^\ast = p(A)q(A)^{-1} +$ low rank and $\text{deg}(q) \geq 2$?}}
% {\color{red}{Comment: This section serves as information for the authors, not necessary to include this in the article!}}

% \bibliographystyle{plain}
%\bibliography{longstrings,paper3}

\end{document}